\documentclass[11pt, letterpaper, oneside,reqno]{amsart}

\usepackage{latexsym, amsmath, amssymb, amsfonts, amscd,bm}
\usepackage{amsthm}
\usepackage{t1enc}
\usepackage[mathscr]{eucal}
\usepackage{indentfirst}
\usepackage{graphicx, pb-diagram}
\usepackage{fancyhdr}
\usepackage{fancybox}
\usepackage{enumerate}
\usepackage[all]{xy}
\usepackage{hyperref}
\usepackage{tikz-cd}
\usepackage{mathtools}
\usepackage{blkarray}
\usepackage{url}
\usepackage{float}
\usepackage{bm}
\usepackage[makeroom]{cancel}
\usepackage{xcolor}

\theoremstyle{plain}
\newtheorem{thm}{Theorem}[section]

\newtheorem{prop}[thm]{Proposition}

\newtheorem*{goal}{Goal}
\newtheorem*{question}{Question}

\newtheorem{lemma}[thm]{Lemma}
\newtheorem{cor}[thm]{Corollary}

\newtheoremstyle{underline}
{}        
{}              
{}              
{}    
{\large}              
{:}             
{1mm}         
{{\underline{\thmname{#1}\thmnumber{ #2}}}}  

\theoremstyle{underline}
\newtheorem*{claim*}{Claim}

\theoremstyle{definition}
\newtheorem{defi}[thm]{Definition}

\theoremstyle{remark}
\newtheorem{remark}[thm]{Remark}
\newtheorem{ex}[thm]{Example}

\newtheorem*{ack}{Acknowledgements}

\definecolor{forest}{rgb}{0,0.5,0}

\subjclass[2020]{57R30, 53C24, 37C85, 22E15}

\begin{document}
	
	\title{Infinitesimally rigid Lie foliations with dense leaves}
	
	\author{Stephane Geudens}
	\address{{\scriptsize Institute of Mathematics, Polish Academy of Sciences,  ul. Sniadeckich 8, 00-656 Warsaw, Poland}}
	\email{stephane\_geudens@hotmail.com}
	
	\author{Florian Zeiser}
	\address{{\scriptsize Center for Geometry and Physics, Institute for Basic Science, 79 Jigok-ro 127beon-gil, Nam-gu, Pohang-si, Gyeongsangbuk-do, Republic of Korea 37673}}
	\email{fzeiser@ibs.re.kr}

	\begin{abstract}
		We call a foliation $\mathcal{F}$ on a compact manifold infinitesimally rigid if its deformation cohomology $H^{1}(\mathcal{F},N\mathcal{F})$ vanishes. This paper studies infinitesimal rigidity for a distinguished class of Riemannian foliations, namely Lie foliations with dense leaves. We construct infinitesimally rigid Lie foliations with dense leaves, modeled on any compact semisimple Lie algebra with simple ideals different from $\mathfrak{so}(3)$. To our knowledge, these are the first examples of infinitesimally rigid Riemannian foliations that are not Hausdorff.
	\end{abstract}
	
	\maketitle
	
	\setcounter{tocdepth}{1} 
	\tableofcontents
	
	\section{Introduction}
	A foliation $\mathcal{F}$ on a compact manifold $M$ is called rigid if any other foliation sufficiently close to $\mathcal{F}$ in the $C^{\infty}$-topology is conjugate to $\mathcal{F}$ under a diffeomorphism of $M$. Here the $C^{\infty}$-topology on the space of $k$-dimensional foliations is induced by the inclusion 
	\[ \mathrm{Fol}_k(M)\hookrightarrow \Gamma(\mathrm{Gr}_k(M)):\mathcal{F}\mapsto T\mathcal{F},\]
	where $\mathrm{Gr}_k(M)\to M$ is the Grassmannian bundle of $k$-planes on $M$. When studying rigidity of foliations, the first step is understanding the corresponding requirement at the infinitesimal level. This turns out to be the vanishing of  a suitable cohomology group $H^{1}(\mathcal{F},N\mathcal{F})$. The definition of this cohomology group relies on the Bott connection $\nabla$, which is a canonical representation of the Lie algebroid $T\mathcal{F}$ on its normal bundle $N\mathcal{F}$. The Bott  connection gives rise to a complex $\big(\Omega^{\bullet}(\mathcal{F},N\mathcal{F}),d_{\nabla}\big)$ with cohomology groups $H^{\bullet}(\mathcal{F},N\mathcal{F})$. Heitsch \cite{Heitsch} showed that first order deformations of the foliation $\mathcal{F}$, modulo those obtained by applying an isotopy to $\mathcal{F}$, are given by elements in $H^{1}(\mathcal{F},N\mathcal{F})$. It is therefore justified to call $\mathcal{F}$ infinitesimally rigid if $H^{1}(\mathcal{F},N\mathcal{F})$ vanishes. Clearly, infinitesimally rigid foliations are the natural candidates for actual rigidity. This philosophy is embodied by a deep theorem due to Hamilton \cite{Hamilton} which shows that rigidity is implied by a strong version of infinitesimal rigidity, namely the existence of ``tame'' homotopy operators for the complex $\big(\Omega^{\bullet}(\mathcal{F},N\mathcal{F}),d_{\nabla}\big)$.
	
	\vspace{0.2cm}
	
	Riemannian foliations form a distinguished class of foliations which lend themselves to the study of (infinitesimal) rigidity. They are characterized by the existence of a bundle-like Riemannian
	metric \cite{Reinhart}. Equivalently, they are locally defined by Riemannian submersions. On one hand, studying rigidity of Riemannian foliations seems more tractable because of a result by Hamilton \cite{Hamilton}, which outlines a strategy for the construction of tame homotopy operators in the Riemannian setting. The strategy in question applies as soon as one can establish a certain a priori estimate. Hence, this gives a systematic way in which rigidity of Riemannian foliations can be obtained. On the other hand, studying infinitesimal rigidity is more tractable in the Riemannian setting because one has more tools to analyze the cohomology group $H^{1}(\mathcal{F},N\mathcal{F})$. For instance, there is a Hodge decomposition for $\Omega^{\bullet}(\mathcal{F},N\mathcal{F})$ which can be used to study the Hausdorff quotient of $H^{1}(\mathcal{F},N\mathcal{F})$, see \cite{Alvarez}. Another example, which is essential for the present paper, comes from representation theory. Suspension of a discrete group action $\Gamma\rightarrow\text{Diff}(T)$ provides a bridge between group actions and foliation theory, with isometric group actions giving rise to Riemannian foliations. Using this bridge, one can import vanishing results for group cohomology to obtain infinitesimal rigidity of the corresponding suspension foliation. To show that group cohomology vanishes, it is useful to impose that the discrete group $\Gamma$ acts by isometries because this enables one to use tools like the Peter-Weyl theorem. This is exactly how  Lubotzky-Zimmer \cite{Zimmer} obtained an important vanishing result for group cohomology which we will use in this paper.

	\vspace{0.2cm}
	
	Few examples of rigid foliations are known to date. Hamilton \cite{Hamilton} proved a rigidity result for Riemannian foliations with all leaves compact, also called Hausdorff foliations, showing that such a foliation is rigid provided that its generic leaf $L$ satisfies $H^{1}(L)=0$.  Different versions of this result have been obtained independently throughout the years \cite{langevin}, \cite{epstein},\cite{rui}. To our knowledge, these are the only examples of rigid Riemannian foliations known to date. Beyond the Riemannian case, El Kacimi Alaoui and Nicolau \cite{kacimi} established rigidity for a class of foliations obtained by suspending linear foliations on tori with suitable linear Anosov diffeomorphisms. A special case of their result was obtained before by Ghys-Sergiescu \cite{ghys}. Finally, the orbit foliation of the Weyl chamber flow was shown to be infinitesimally rigid by Kanai \cite{Kanai} and Konenko \cite{Konenko}, while rigidity was established by Katok-Spatzier \cite{Katok}. For a survey on (infinitesimal) rigidity of foliations, we refer the reader to \cite[\S\!1.5.1 and \S\!1.5.2]{Asaoka}.

	
	\vspace{0.2cm}
	
	Our paper is motivated by Hamilton's work \cite{Hamilton}. While Hamilton could only apply his method to the class of Hausdorff foliations -- i.e. Riemannian foliations with all leaves compact -- he suggests there should be more examples of rigid Riemannian foliations. We aim to make a contribution in this direction by constructing the first examples of infinitesimally rigid Riemannian foliations that are not Hausdorff. Rather than considering arbitrary Riemannian foliations, we will focus on the distinguished subclass of Lie foliations. This is justified by Molino's structure theory \cite{Molino}, which reduces the study of Riemannian foliations on compact manifolds to that of Lie foliations with dense leaves. From this point of view, Lie foliations with dense leaves are the most elementary kind of Riemannian foliations. 
	
	\begin{goal}
	Construct infinitesimally rigid Lie foliations with dense leaves on compact manifolds.
	\end{goal}
	
	In \S\ref{sec:two}, we review some background about infinitesimal rigidity and Lie foliations. Let us only recall here that a Lie foliation $\mathcal{F}$ is a foliation whose transversal structure is modeled on a Lie group. Equivalently, the tangent distribution of $\mathcal{F}$ is given by the kernel of a non-singular, Lie algebra valued one-form $\omega\in\Omega^{1}(M,\mathfrak{g})$ satisfying the Maurer-Cartan equation. The prototypical example is the foliation defined by a flat connection on a principal bundle.
	
	\vspace{0.2cm}
	
	In \S\ref{sec:three}, we initiate our study of infinitesimally rigid Lie foliations with dense leaves. Given a Lie $\mathfrak{g}$-foliation $\mathcal{F}$ with dense leaves on $M$, we look for properties that the Lie algebra $\mathfrak{g}$ and the manifold $M$ need to satisfy for $\mathcal{F}$ to be infinitesimally rigid (see Prop.~\ref{prop:obstruction} and Cor.~\ref{cor:not-amenable}). 
	
	\vspace{0.2cm}
	\noindent
	\textbf{Theorem A.} \emph{Let $\mathcal{F}$ be an infinitesimally rigid Lie $\mathfrak{g}$-foliation with dense leaves on a connected, compact manifold $M$. Then the following hold:
		\begin{itemize}
			\item $\mathfrak{g}$ is perfect,
			\item $H^1(M)$ vanishes,
			\item $\pi_1(M)$ is not amenable.
		\end{itemize}
	}
	
	\vspace{0.2cm}
	
	In \S\ref{sec:four}, we prove a result that can be used to construct a new infinitesimally rigid Lie foliation with dense leaves out of a given one. If $\mathcal{F}$ is a Lie $\mathfrak{g}$-foliation defined by $\omega\in\Omega^{1}(M,\mathfrak{g})$ and $\mathfrak{h}\subset\mathfrak{g}$ is an ideal, then dividing out the ideal $\mathfrak{h}$ gives a Lie $\mathfrak{g}/\mathfrak{h}$-foliation $\overline{\mathcal{F}}$ defined by $\overline{\omega}\in\Omega^{1}(M,\mathfrak{g}/\mathfrak{h})$. Note that the leaves of $\mathcal{F}$ are contained in the leaves of $\overline{\mathcal{F}}$. If $\mathcal{F}$ has dense leaves, then infinitesimal rigidity is preserved under this reduction procedure provided that the ideal $\mathfrak{h}$ is perfect (see Prop.~\ref{prop:spectral} and Cor.~\ref{cor:semi-simple}).
	
	\vspace{0.2cm}
	\noindent
	\textbf{Theorem B.} \emph{Let $\mathcal{F}$ be a Lie $\mathfrak{g}$-foliation with dense leaves on a manifold $M$. A choice of ideal $\mathfrak{h}\subset\mathfrak{g}$ gives a Lie  $\mathfrak{g}/\mathfrak{h}$-foliation $\overline{\mathcal{F}}$ which also has dense leaves. If $\mathcal{F}$ is infinitesimally rigid and $\mathfrak{h}$ is perfect, then $\overline{\mathcal{F}}$ is also infinitesimally rigid.
	}
	
	\vspace{0.2cm}
	\noindent
	The proof uses a version of the Serre spectral sequence for Lie algebroid cohomology recently developed in \cite{Serre}. Thm.~B allows one to pass from an infinitesimally rigid Lie foliation with dense leaves modeled on a semisimple Lie algebra to one modeled on a simple Lie algebra. 
	
	\vspace{0.2cm}
	The core of this note is \S\ref{sec:five}, where we construct infinitesimally rigid Lie foliations with dense leaves on compact manifolds. To our knowledge, these are the first examples of infinitesimally rigid Riemannian foliations that are not Hausdorff. Our examples are obtained via the suspension method, which we recall in \S\ref{sec:suspension}. We prove the following (see Thm.~\ref{thm:main-result}). 
	
	\vspace{0.2cm}
	\noindent
	\textbf{Theorem C.} \emph{Let $B$ be a connected, compact manifold such that $\pi_1(B)$ has property (T) as a discrete group. Let $G$ be a connected, compact Lie group and assume that $\varphi:\pi_1(B)\rightarrow G$ is a group homomorphism such that $\varphi(\pi_1(B))\subset G$ is dense. Then the suspension of $\varphi$ gives a Lie $\mathfrak{g}$-foliation $\mathcal{F}$ with dense leaves on a compact manifold satisfying $H^{1}(\mathcal{F},N\mathcal{F})=0$.}
	
	\vspace{0.2cm}
	\noindent
	The proof of Thm.~C relies on two auxiliary results, which we now turn to describe. The main source of difficulty in finding infinitesimally rigid foliations is the lack of general methods to compute $H^{1}(\mathcal{F},N\mathcal{F})$. Recent work by El Kacimi Alaoui \cite{Kacimi} addresses this issue by computing $H^{1}(\mathcal{F},N\mathcal{F})$ for developable foliations $\mathcal{F}$ -- these are foliations whose lift to some covering of $M$ is given by a fibration. The suspension foliation $\mathcal{F}$ of a homomorphism $\varphi:\pi_{1}(B)\rightarrow G$ is of this type. Hence, the problem is reduced to requiring that the group cohomology $H^{1}(\pi_1(B),C^{\infty}(G))$ vanishes. This in turn is achieved by applying a cohomology vanishing result due to Lubotzky-Zimmer \cite{Zimmer} which applies to isometric, ergodic actions of discrete groups that have Kazhdan's property (T). The result is Thm.~C above.
	
	\vspace{0.2cm}
	We then proceed by constructing explicit examples of Thm.~C. It follows from Thm.~A that the Lie group $G$ is necessarily semisimple, and a result by Zimmer \cite{zimmer-actions} implies that $G$ cannot admit $SO(3)$ as a quotient. It turns out that the construction in Thm.~C can be carried out whenever $G$ satisfies these assumptions, which yields the following (see Cor.~\ref{cor:existence}).
	
	\vspace{0.2cm}
	\noindent
	\textbf{Theorem D.} \emph{Let $\mathfrak{g}$ be a compact, semisimple Lie algebra with no simple factor isomorphic to $\mathfrak{so}(3)$. Then there exists an infinitesimally rigid Lie $\mathfrak{g}$-foliation $\mathcal{F}$ with dense leaves on some compact manifold $M$.}
	
	\vspace{0.2cm}
	\noindent
	We would like to stress that Thm.~D is not just an abstract existence result -- we also provide an explicit algorithm describing how these foliations can be obtained. It works as follows.
	
	\vspace{0.15cm}
	
	\emph{Input:} A compact, connected, semisimple Lie group $G$ not admitting $SO(3)$ as a quotient.
	
	\emph{Output:} A finitely presented, dense subgroup $\Gamma\subset G$ that is a discrete Kazhdan group.
	
	\vspace{0.15cm}
	\noindent 
	Since $\Gamma$ is finitely presented, one can fix a compact connected $4$-manifold $B$ with $\pi_1(B)=\Gamma$. The suspension of the inclusion $\pi_1(B)\hookrightarrow G$ is then the desired infinitesimally rigid Lie $\mathfrak{g}$-foliation. Our algorithm is a variation on a result by de Cornulier \cite{Cornulier} determining which connected Lie groups $G$ admit a finitely generated, dense subgroup that has property (T) as a discrete group. Our proof uses similar techniques. The difference is that we only work with compact semisimple Lie groups, which allow for a more concrete approach, and that the desired subgroup $\Gamma$ should be finitely presented rather than merely finitely generated.
	
	We now give a rough overview of the algorithm, ignoring some of the technical details. We first reduce to the case in which $G$ is compact, connected and simple, not locally isomorphic to $SO(3)$ and with trivial center. Then $G$ can be realized as a closed subgroup of $SL(n,\mathbb{R})$ that is defined over $\mathbb{Q}$ as an algebraic group. Fix a number field $K\subset\mathbb{R}$ of degree $3$ that is not totally real and set $\Gamma:=G(\mathcal{O}_{K})$. That is, $\Gamma$ consists of the matrices in $G$ whose entries lie in the ring of integers $\mathcal{O}_{K}\subset K$. It seems difficult a priori to check that $\Gamma$ satisfies the requirements. To facilitate this, we make use of a result by Borel-Harish-Chandra \cite{harish} which yields a diagonal embedding $(\text{Id},\sigma):\Gamma\hookrightarrow G\times G_{\mathbb{C}}$ such that $(\text{Id},\sigma)(\Gamma)\subset G\times G_{\mathbb{C}}$ is a cocompact, irreducible lattice. This enables us to use lattice theory in order to show that $\Gamma$ has the desired properties. First, the fact that $\Gamma$ is finitely presented follows from cocompactness of the lattice $(\text{Id},\sigma)(\Gamma)\subset G\times G_{\mathbb{C}}$. Second, the fact that $\Gamma\subset G$ is dense follows from irreducibility of the lattice $(\text{Id},\sigma)(\Gamma)\subset G\times G_{\mathbb{C}}$. Finally, $\Gamma$ has property (T) as a discrete group because the lattice $(\text{Id},\sigma)(\Gamma)$ inherits property (T) from $G\times G_{\mathbb{C}}$.
	
	\vspace{0.2cm}
	
	We finish the paper by working out a concrete example of Thm.~C in \S\ref{sec:example}. The example in question is an infinitesimally rigid Lie foliation with dense leaves modeled on $\mathfrak{so}(n)$ for $n\geq 5$. Its construction deviates slightly from the above algorithm, but it uses similar ideas. We like to mention it explicitly because the dense subgroup $\Gamma\subset SO(n)$ with property (T) used in its construction has appeared before in a different context. It was used in Sullivan's solution of the Ruziewicz problem for $n\geq 4$, showing that the Lebesgue measure is the only finitely additive measure defined on all Lebesgue measurable subsets of the sphere $S^{n}$ that is invariant under the action of $SO(n+1)$ and of total measure $1$, see \cite[\S~3.4]{Lubotzky}.
	
	\vspace{0.2cm}
	
	This paper raises a few interesting open questions. First, it is natural to wonder whether the infinitesimally rigid foliations we constructed in Thm.~C are in fact rigid. To establish rigidity, one could try to apply the Nash-Moser method, which requires the construction of tame homotopy operators for the deformation complex of the foliation. The main challenge will consist of finding a way to exploit the various assumptions in Thm.~C when constructing these homotopy operators. Second, the particular construction described in Thm.~C has the limitation that it only produces examples of infinitesimally rigid Lie foliations with dense leaves for compact, semisimple Lie algebras with simple ideals different from $\mathfrak{so(3)}$. It would be interesting to know if a different construction can yield examples modeled on perfect Lie algebras that are not of this type.

	\color{black}	
	\begin{ack}
		S.G. acknowledges support from the UCL Institute for Mathematical and Statistical Sciences (IMSS) and from the Institute of Mathematics of the Polish Academy of Sciences (IMPAN). F.Z. acknowledges support from UIUC and the Institute for Basic Science (IBS) in Pohang.  We would like to thank Ioan M\u{a}rcu\c{t} for useful discussions, and Aziz El Kacimi Alaoui for helpful e-mail exchanges about his paper \cite{Kacimi}.
	\end{ack}

	\section{Background on Lie foliations and infinitesimal rigidity}\label{sec:two}
	We recall some definitions concerning the main goal of this paper, namely the construction of foliations that are infinitesimally rigid. Since we will focus our attention on the class of Lie foliations, we also recall their definition and some of their properties.
	
	
	\subsection{Infinitesimal rigidity of foliations}
	Let $M$ be a compact manifold with an arbitrary foliation $\mathcal{F}$. For a vector field $Y\in\mathfrak{X}(M)$, we denote by $\overline{Y}$ the induced section of the normal bundle $N\mathcal{F}:=TM/T\mathcal{F}$. It is well-known that there is a canonical representation $\nabla$ of the Lie algebroid $T\mathcal{F}$ on the normal bundle $N\mathcal{F}$, given by
	\[
	\nabla_{X}\overline{Y}=\overline{[X,Y]},
	\]
	for $X\in\Gamma(T\mathcal{F})$ and $\overline{Y}\in\Gamma(N\mathcal{F})$. This representation $\nabla$ is called the Bott connection. It induces a differential $d_{\nabla}$ on the graded vector space $\Omega^{\bullet}(\mathcal{F},N\mathcal{F}):=\Gamma(\wedge^{\bullet}T^{*}\mathcal{F}\otimes N\mathcal{F})$ of foliated forms with coefficients in $N\mathcal{F}$, given by the usual Koszul formula
	\begin{align*}
		d_{\nabla}\eta(V_1,\ldots,V_{k+1})&=\sum_{i=1}^{k+1}(-1)^{i+1}\nabla_{V_i}\big(\eta(V_1,\ldots,V_{i-1},\widehat{V_i},V_{i+1},\ldots,V_{k+1})\big)\\
		&\hspace{0.5cm}+\sum_{i<j}(-1)^{i+j}\eta\big([V_i,V_j],V_1,\ldots,\widehat{V_i},\ldots,\widehat{V_j},\ldots,V_{k+1}\big).
	\end{align*}
	We denote the associated cohomology groups by $H^{\bullet}(\mathcal{F},N\mathcal{F})$. Heitsch \cite{Heitsch} showed that the first cohomology $H^{1}(\mathcal{F},N\mathcal{F})$ governs infinitesimal deformations of $\mathcal{F}$, in the following sense.
	
	Let $\mathcal{F}_{t}$ be a smooth path of foliations on $M$ with $\mathcal{F}_{0}=\mathcal{F}$. We identify the normal bundle $N\mathcal{F}$ with a subbundle of $TM$ that is complementary to $T\mathcal{F}$. 
	Because $M$ is compact, there exists $\epsilon>0$ such that $T\mathcal{F}_t$ remains transverse to $N\mathcal{F}$ for $0\leq t\leq\epsilon$. Hence, there exists a path $\eta_t\in\Gamma(T^{*}\mathcal{F}\otimes N\mathcal{F})$ such that for all $0\leq t\leq\epsilon$ we have
	\[
	T\mathcal{F}_t=\text{graph}(\eta_t)=\{X+\eta_t(X): X\in\Gamma(T\mathcal{F})\}.
	\]
	
	\begin{lemma}[\cite{Heitsch}]
		\label{lem:first-order-fol}
		Assuming the setup just described, we have:
		\begin{enumerate}
			\item The infinitesimal deformation $\left.\frac{d}{dt}\right|_{t=0}\eta_t$ is closed with respect to $d_{\nabla}$.
			\item If the path $\mathcal{F}_t$ is generated by an isotopy $\phi_t$, i.e. $T\mathcal{F}_t=(\phi_t)_{*}T\mathcal{F}$, then the corresponding infinitesimal deformation is exact. Indeed,
			\[
			\left.\frac{d}{dt}\right|_{t=0}\eta_t=d_{\nabla}\overline{V_0},
			\]
			where $V_t$ is the time-dependent vector field corresponding with the isotopy $\phi_t$.
		\end{enumerate}
	\end{lemma}
	
	This result justifies the following definition.
	
	\begin{defi}
		We call the foliation $\mathcal{F}$ infinitesimally rigid if $H^{1}(\mathcal{F},N\mathcal{F})$ vanishes.
	\end{defi}
	
	The aim of this note is to construct foliations on compact manifolds that are infinitesimally rigid. This is not a straightforward task, since the cohomology group $H^{1}(\mathcal{F},N\mathcal{F})$ is hard to compute due to lack of general methods. The problem becomes more manageable if one assumes that the foliation $\mathcal{F}$ is transversely parallelizable, meaning that there is a frame $\{\overline{Y_1},\ldots,\overline{Y_q}\}$ for the normal bundle $N\mathcal{F}$ consisting of sections $\overline{Y_i}$ that are flat with respect to the Bott connection $\nabla$. In that case, we get 
	\[
	H^{1}(\mathcal{F},N\mathcal{F})\cong H^{1}(\mathcal{F})\otimes\mathbb{R}^{q},
	\]
	hence the problem is reduced to computing the ordinary foliated cohomology $H^1(\mathcal{F})$. In this note, we are interested in a particular class of transversely parallelizable foliations called Lie foliations. We introduce them in the next subsection.

	\subsection{Lie foliations}
	We briefly recall the definition, as well as Fédida's structure theory for Lie foliations.  For a more detailed account, we refer to \cite[\S 4.3.1]{Moerdijk}.
	
	\begin{defi}\label{def:Lie}
		Let $M$ be a manifold and $\mathfrak{g}$ a Lie algebra. A Maurer-Cartan form with values in $\mathfrak{g}$ is a differential form $\omega\in\Omega^{1}(M,\mathfrak{g})$ satisfying the Maurer-Cartan equation
		\begin{equation}\label{eq:MC}
			d\omega+\frac{1}{2}[\omega,\omega]=0.
		\end{equation}
		Assume moreover that $\omega$ is non-singular, i.e. $\omega_x:T_{x}M\rightarrow\mathfrak{g}$ is surjective for all $x\in M$. Then  $\ker\omega$ integrates to a foliation $\mathcal{F}$, which we call the Lie $\mathfrak{g}$-foliation defined by $\omega\in\Omega^{1}(M,\mathfrak{g})$.
	\end{defi}
	
	In the above definition, $d\omega$ is defined by picking any basis of $\mathfrak{g}$ and then applying the de Rham differential to each component of $\omega$. The two-form $[\omega,\omega]\in \Omega^{2}(M,\mathfrak{g})$ is given by
	\[
	[\omega,\omega](X,Y)=2[\omega(X),\omega(Y)].
	\]
	
	\begin{remark}\label{remark: trivialization}
		It is classical that Lie foliations are transversely parallelizable \cite[Prop.~4.21]{Moerdijk}. Note that the Maurer-Cartan form $\omega\in\Omega^{1}(M,\mathfrak{g})$ induces a pointwise isomorphism
		\[
		\omega_x:N_{x}\mathcal{F}\rightarrow\mathfrak{g}.
		\]
		Picking a basis $e_1,\ldots,e_q$ of $\mathfrak{g}$, we can define sections $\overline{Y_1},\ldots,\overline{Y_q}$ of $N\mathcal{F}$ by requiring that
		\[
		\overline{Y_i}(x)=\omega_{x}^{-1}(e_i).
		\]
		These constitute a transverse parallelism for $\mathcal{F}$. In particular, for a Lie foliation $\mathcal{F}$ modeled on a $q$-dimensional Lie algebra, we have that $H^{1}(\mathcal{F},N\mathcal{F})\cong H^{1}(\mathcal{F})\otimes\mathbb{R}^{q}$.
	\end{remark}
	
	We now recall some structure theory of Lie foliations. Let $\mathcal{F}$ be a Lie foliation defined by a Maurer-Cartan form $\omega\in\Omega^{1}(M,\mathfrak{g})$ on a compact, connected manifold $M$. Denote by $G$ the unique connected and simply connected Lie group integrating $\mathfrak{g}$. Then there exist a normal covering $\pi:\widehat{M}\rightarrow M$ and a fiber bundle map $D:\widehat{M}\rightarrow G$ whose fibers coincide with the leaves of the lifted foliation $\pi^{*}\mathcal{F}$ on $\widehat{M}$. The covering space $\widehat{M}$ is called the Darboux cover of $\omega$, and $D$ is called its developing map.
	The action of $\pi_1(M)$ on $\widehat{M}$ by covering transformations defines a group homomorphism
	\begin{equation*}\label{eq:hol}
		h:\pi_1(M)\rightarrow G,
	\end{equation*}
	and the fibration $D:\widehat{M}\rightarrow G$ is equivariant with respect to $h$ in the sense that
	\[
	D(\widehat{x}\gamma)=D(\widehat{x})h(\gamma)
	\]
	for $\widehat{x}\in\widehat{M}$ and $\gamma\in\pi_1(M)$. The image $H:=h(\pi_1(M))$ is called the holonomy group of $\omega$. 
	There is a manifold structure on the set of left cosets $G/\overline{H}$ such that the projection map $G\rightarrow G/\overline{H}$ is a submersion. The developing map $D:\widehat{M}\rightarrow G$ then descends to a fibration $M\rightarrow G/\overline{H}$, which is called the basic fibration. In summary, we have the following diagram
	\begin{equation}\label{eq:diagram}
	\begin{tikzcd}
		&(\widehat{M},\pi^{*}\mathcal{F})\arrow[r,"D"]\arrow[d,"\pi"]	&G\arrow[d]\\
		&(M,\mathcal{F})\arrow[r]& G/\overline{H}
	\end{tikzcd}.
	\end{equation}
	We highlight two consequences of this diagram which will be used in what follows.
	\begin{enumerate}[i)] 
	\item The leaves of $\mathcal{F}$ are dense exactly when $H\subset G$ is a dense subgroup  \cite[Lemma~4.23]{Moerdijk}.
	\item We get a useful description for the space of basic differential forms
	\[
	\Omega^{\bullet}_{bas}(M)=\left\{\alpha\in\Omega^{\bullet}(M):\iota_V\alpha=\iota_V d\alpha=0,\ \ \forall V\in\Gamma(T\mathcal{F})\right\}.
	\]
	Indeed, a basic form $\alpha\in \Omega^{k}_{bas}(M)$ pulls back to a form $\pi^{*}\alpha\in\Omega^{k}(\widehat{M})$ that is basic for the foliation $\pi^{*}\mathcal{F}$, which is given by the fibers of the developing map $D:\widehat{M}\rightarrow G$. Since $\pi^{*}\alpha$ is moreover invariant under covering transformations, it descends to a form $\overline{\alpha}\in\Omega^{k}(G)$ that is invariant under right translations by elements of $H$. This way, we obtain an isomorphism of complexes
	\[
	\left(\Omega^{\bullet}_{bas}(M),d\right)\overset{\sim}{\longrightarrow}\left(\Omega_{H}^{\bullet}(G),d\right).
	\]
	If the leaves of $\mathcal{F}$ are dense -- i.e. $H$ is a dense subgroup of $G$ -- we obtain the complex of right invariant differential forms on $G$. This coincides with the complex of left invariant differential forms on the opposite Lie group $\overline{G}$ with multiplication $g\overline{\cdot}h =hg$. The latter is isomorphic to the Chevalley-Eilenberg complex of the opposite Lie algebra $(\mathfrak{g},-[\cdot,\cdot])$, which in turn is isomorphic to the Chevalley-Eilenberg complex of the original Lie algebra $(\mathfrak{g},[-,-])$. In particular, in case $\mathcal{F}$ has dense leaves we get that
	\begin{equation}\label{eq:bas}
	H^{1}_{bas}(M)\cong H^{1}(\mathfrak{g})=(\mathfrak{g}/[\mathfrak{g},\mathfrak{g}])^{*}.
	\end{equation}
	 We can realize this isomorphism even more concretely. First, observe that applying an element $\xi\in(\mathfrak{g}/[\mathfrak{g},\mathfrak{g}])^{*}$ to $\omega\in\Omega^{1}(M,\mathfrak{g})$ gives a one-form $\xi\circ\omega\in\Omega^{1}(M)$ that is closed because of the Maurer-Cartan equation \eqref{eq:MC}. Moreover, it is clear that $\xi\circ\omega\in\Omega^{1}_{bas}(M)$. 
	We thus obtain a linear map
	\begin{equation}\label{eq:psi}
		\Psi_{bas}:(\mathfrak{g}/[\mathfrak{g},\mathfrak{g}])^{*}\longrightarrow H^{1}_{bas}(M):\xi\mapsto [\xi\circ\omega].
	\end{equation}
	We claim that this map is an isomorphism. We already know from \eqref{eq:bas} that the dimensions of $(\mathfrak{g}/[\mathfrak{g},\mathfrak{g}])^{*}$ and $H^{1}_{bas}(M)$ agree, hence we only need to check injectivity. To do so, let $\xi\in(\mathfrak{g}/[\mathfrak{g},\mathfrak{g}])^{*}$ be such that $[\xi\circ\omega]\in H^{1}_{bas}(M)$ is trivial. Because the leaves of $\mathcal{F}$ are dense, this means that $\xi\circ\omega$ vanishes identically. Since $\omega_x:T_{x}M\rightarrow\mathfrak{g}$ is surjective for all $x\in M$, this means that $\xi=0$. Hence, the map $\Psi_{bas}$ is an isomorphism.
	\end{enumerate}
	

	\color{black}
	\section{Obstructions for infinitesimal rigidity of Lie foliations}\label{sec:three}
	From now on, we focus our attention on Lie $\mathfrak{g}$-foliations $\mathcal{F}$ with dense leaves on compact, connected manifolds $M$. We investigate which properties the Lie algebra $\mathfrak{g}$ and the manifold $M$ need to satisfy for $\mathcal{F}$ to be infinitesimally rigid. We obtain two types of obstructions: one is cohomological, while the other concerns amenability of the fundamental group $\pi_{1}(M)$.

	\subsection{A cohomological obstruction}
	The first obstruction involves the de Rham cohomology of the manifold $M$ and the Chevalley-Eilenberg cohomology of the Lie algebra $\mathfrak{g}$.

	\begin{prop}\label{prop:obstruction}
		Let $\mathcal{F}$ be a Lie $\mathfrak{g}$-foliation on a compact, connected manifold $M$ with defining Maurer-Cartan form $\omega\in\Omega^{1}(M,\mathfrak{g})$. Assume that $\mathcal{F}$ has dense leaves. If $\mathcal{F}$ is infinitesimally rigid, then $\mathfrak{g}$ is perfect and $H^1(M)$ vanishes. 
	\end{prop}
	
	Vanishing of $H^{1}(M)$ actually implies that the Lie algebra $\mathfrak{g}$ is perfect, as we explain now. First recall that for any foliation $\mathcal{F}$, we have a natural embedding $H^{1}_{bas}(M)\hookrightarrow H^{1}(M)$. See for instance \cite[Prop.~4.1]{Tondeur}. Next, we use the isomorphism $H^{1}_{bas}(M)\cong H^{1}(\mathfrak{g})$ described in \eqref{eq:psi} above, which uses the assumption that $\mathcal{F}$ is a Lie $\mathfrak{g}$-foliation with dense leaves. In the conclusion of Prop.~\ref{prop:obstruction}, we prefer to state explicitly that $\mathfrak{g}$ is perfect for the sake of clarity.
	
	\begin{proof}[Proof of Prop.~\ref{prop:obstruction}]
By the above, we only need to show that $H^{1}(M)$ vanishes. Let us denote by $d$ the dimension of $H^{1}(M)$. To show that necessarily $d=0$, we proceed in three steps.

\vspace{0.2cm}
\noindent
\underline{Step 1:} The natural embedding $H^{1}_{bas}(M)\hookrightarrow H^{1}(M)$ is in fact an isomorphism.

\vspace{0.1cm}
The proof of this fact uses the spectral sequence $(E_k,d_k)$ associated with the foliation $\mathcal{F}$, which is recalled in \cite[Chapter 4]{Tondeur}. Alternatively, this spectral sequence is a special case of the one we will use in \S\ref{sec:four}, so one can also look there for an explicit description of $(E_k,d_k)$.  Since the spectral sequence of $\mathcal{F}$ converges to the de Rham cohomology $H^{\bullet}(M)$, we have 
\begin{equation}\label{eq:sum-spectral}
H^{1}(M)\cong E^{1,0}_{\infty}\oplus E^{0,1}_{\infty}.
\end{equation}
We first study the limiting term $E^{0,1}_{\infty}$. It is well-known that $E^{0,1}_{1}\cong H^{1}(\mathcal{F})$, which vanishes because $\mathcal{F}$ is transversely parallelizable and infinitesimally rigid. Hence also $E^{0,1}_{\infty}$ vanishes. Concerning the limiting term $E^{1,0}_{\infty}$, it is well-known that $E^{1,0}_{2}\cong H^{1}_{bas}(M)$. Because the differential $d_k$ has bi-degree $(k,1-k)$, it then follows that also $E^{1,0}_{\infty}\cong H^{1}_{bas}(M)$. Hence, the isomorphism \eqref{eq:sum-spectral} just states that $H^{1}(M)$ and $H^{1}_{bas}(M)$ are isomorphic. This proves Step 1. 
	
\vspace{0.2cm}
\noindent
\underline{Step 2:} There exists a basis $[\alpha_1],\ldots,[\alpha_d]$ of $H^{1}(M)$ represented by closed basic one-forms $\alpha_1,\ldots,\alpha_{d}\in\Omega^{1}_{bas}(M)$ that are linearly independent at every point of $M$. 

\vspace{0.1cm}
Composing the isomorphism \eqref{eq:psi} with the isomorphism $H^{1}_{bas}(M)\hookrightarrow H^{1}(M)$ from Step 1, we obtain another isomorphism which we denote by
\[
\Psi:(\mathfrak{g}/[\mathfrak{g},\mathfrak{g}])^{*}\longrightarrow H^{1}(M):\xi\mapsto [\xi\circ\omega].
\]  
This isomorphism tells us how to prove the statement of Step 2. Picking linearly independent $\xi_1,\ldots,\xi_d\in(\mathfrak{g}/[\mathfrak{g},\mathfrak{g}])^{*}$, we obtain closed basic one-forms $\xi_1\circ\omega,\ldots,\xi_d\circ\omega\in\Omega^{1}_{bas}(M)$ whose cohomology classes $[\xi_1\circ\omega],\ldots,[\xi_d\circ\omega]$ form a basis of $H^{1}(M)$. It remains to check that $\xi_1\circ\omega,\ldots,\xi_d\circ\omega$ are linearly independent at every point $x\in M$. To do so, assume that
\[
c_1(\xi_1\circ\omega)_x+\ldots+c_d(\xi_d\circ\omega)_x=0
\]
for some $c_1,\ldots,c_d\in\mathbb{R}$. Clearly, this is equivalent with
\[
\left(\sum_{i=1}^{d}c_i\xi_i\right)\circ\omega_x=0.
\]
Because the map $\omega_x:T_{x}M\rightarrow\mathfrak{g}$ is surjective, this implies that $\sum_i c_i\xi_i=0$. Since we chose $\xi_1,\ldots,\xi_d$ to be linearly independent, we get that $c_1=\cdots=c_d=0$. This proves Step 2.

\vspace{0.2cm}
\noindent
\underline{Step 3:} We deduce that $d=0$ and therefore $H^{1}(M)$ vanishes.

\vspace{0.1cm}
Consider the basic one-forms $\alpha_1,\ldots,\alpha_d\in\Omega^{1}_{bas}(M)$ constructed in Step 2. Since they are closed and linearly independent at every point, they give rise to a foliation $\mathcal{G}$ defined by
\[
T\mathcal{G}=\ker(\alpha_1)\cap\cdots\cap\ker(\alpha_d).
\]
We claim that $\mathcal{G}$ is given by the fibers of a fiber bundle $M\rightarrow\mathbb{T}^{d}$. Assuming this claim for now, we show that the conclusion of Step 3 holds. Since the one-forms $\alpha_1,\ldots,\alpha_d$ are basic, we know that $T\mathcal{F}\subset T\mathcal{G}$ and therefore the leaves of $\mathcal{F}$ are contained in the fibers of $M\rightarrow\mathbb{T}^{d}$. The same then holds for the leaf closures of $\mathcal{F}$. Since $\mathcal{F}$ has dense leaves by assumption, it follows that the base $\mathbb{T}^{d}$ is zero-dimensional. Hence $d=0$ and therefore $H^{1}(M)$ vanishes.

It remains to argue that $\mathcal{G}$ is given by the fibers of a fiber bundle $M\rightarrow\mathbb{T}^{d}$. To do so, note that $\mathcal{G}$ is itself a Lie foliation modeled on the abelian Lie algebra $\mathbb{R}^{d}$. If we show that its holonomy group $H\subset\mathbb{R}^{d}$ is a full rank lattice, then the diagram \eqref{eq:diagram} implies that the leaves of $\mathcal{G}$ are given by the fibers of the basic fibration $M\rightarrow\mathbb{T}^{d}$. The holonomy homomorphism $h:\pi_{1}(M)\rightarrow\mathbb{R}^{d}$ of $\mathcal{G}$ factors through the abelianization $H_{1}(M,\mathbb{Z})$ and is simply given by
\[
h:H_{1}(M,\mathbb{Z})\rightarrow\mathbb{R}^{d}:[\sigma]\mapsto\left( \int_{\sigma}\alpha_1,\ldots,\int_{\sigma}\alpha_d\right).
\]
See for instance \cite[Ex.~4.22]{Moerdijk}. Pick generators $[\sigma_1],\ldots,[\sigma_d]$ for the free part of $H_{1}(M,\mathbb{Z})$. To show that the holonomy group $H\subset\mathbb{R}^{d}$ is a full rank lattice, we check that the rows of
\begin{equation}\label{eq:matrix}
\begin{pmatrix}
\int_{\sigma_1}\alpha_1 & \cdots & \int_{\sigma_1}\alpha_d \\
\vdots & & \vdots \\
\int_{\sigma_d}\alpha_1 & \cdots & \int_{\sigma_d}\alpha_d
\end{pmatrix}
\end{equation}
are linearly independent. But this is an immediate consequence of the de Rham isomorphism
\[
H^{1}(M)\overset{\sim}{\longrightarrow}\text{Hom}\big(H_{1}(M,\mathbb{Z}),\mathbb{R}\big):[\alpha]\mapsto\left([\sigma]\mapsto\int_{\sigma}\alpha\right).
\]
Indeed, because $[\alpha_1],\ldots,[\alpha_d]$ is a basis of $H^{1}(M)$, the de Rham isomorphism implies that the columns of the matrix \eqref{eq:matrix} are linearly independent. It then follows that also its rows are linearly independent. This confirms that the holonomy group $H\subset\mathbb{R}^{d}$ is a full rank lattice, hence $\mathcal{G}$ is given by the fibers of the basic fibration $M\rightarrow\mathbb{T}^{d}$. This finishes the proof.
\end{proof}

\color{black}	
	\begin{cor}\label{cor:semisimple}
		Let $M$ be a compact, connected manifold and $\mathcal{F}$ a Lie $\mathfrak{g}$-foliation with dense leaves. Assume that $\mathcal{F}$ is infinitesimally rigid. If $\mathfrak{g}$ is compact, then it is semisimple.
	\end{cor}
	\begin{proof}
		Because compact Lie algebras are reductive, we know that $[\mathfrak{g},\mathfrak{g}]$ is semisimple. Since $\mathfrak{g}=[\mathfrak{g},\mathfrak{g}]$ by Prop.~\ref{prop:obstruction}, it follows that $\mathfrak{g}$ is semisimple.
	\end{proof}
	
	\subsection{Non-amenability of the fundamental group}
	
	Recall that a discrete group $\Gamma$ is amenable if there is a finitely additive left invariant probability measure $\mu:\mathcal{P}(\Gamma)\rightarrow[0,1]$. Here $\mathcal{P}(\Gamma)$ is the power set of $\Gamma$. For more details on amenability, we refer to \cite{Juschenko}. Examples of amenable groups are finite, abelian and solvable groups. The standard example of a non-amenable group is the free group on two generators $\mathbb{F}_{2}$. If a discrete group $\Gamma$ is amenable, then so are the subgroups of $\Gamma$ and the quotients of $\Gamma$ by normal subgroups of $\Gamma$.
	
	\begin{cor}\label{cor:not-amenable}
		Let $M$ be a compact, connected manifold and $\mathcal{F}$ a Lie $\mathfrak{g}$-foliation on $M$ with dense leaves. If $\mathcal{F}$ is infinitesimally rigid, then $\pi_1(M)$ is not amenable.
	\end{cor}
	\begin{proof}
		Let $G$ be the connected, simply connected Lie group integrating $\mathfrak{g}$. Since $\mathfrak{g}$ is perfect by Prop.~\ref{prop:obstruction}, it follows that $G$ is not solvable. The holonomy subgroup $H=h\big(\pi_1(M)\big)\subset G$ is finitely generated and dense, therefore it contains a free subgroup $\mathbb{F}_{2}$ by \cite[Thm.~1.3]{Breuillard}. This implies that $\pi_1(M)$ is not amenable, for otherwise every subgroup of $H$ -- including this copy of $\mathbb{F}_2$ -- would be amenable. This proves the statement.
	\end{proof}
	
	\begin{remark}
		The result \cite[Thm.~1.3]{Breuillard} used in the proof above follows from an enhancement of the classical Tits alternative, see \cite[Thm.~4.5]{Breuillard} and \cite[Rem.~4.6]{Breuillard}. The Tits alternative itself actually suffices to prove Cor.~\ref{cor:not-amenable}. We give a sketch of the argument.
		
		Assume by contradiction that $\pi_1(M)$ is amenable. Then also the holonomy group $H\subset G$ is amenable. Taking its image under the adjoint representation $\text{Ad}:G\rightarrow\text{GL}(\mathfrak{g})$, it follows that $\text{Ad}(H)\subset\text{GL}(\mathfrak{g})$ is amenable. The Tits alternative \cite[Thm.~1]{Tits} says that $\text{Ad}(H)$ either has a non-abelian free subgroup or a solvable subgroup of finite index. The first option is impossible since $\text{Ad}(H)$ is amenable, hence there must exist a solvable subgroup $\Gamma\subset\text{Ad}(H)$ of finite index. Then  $H_{\Gamma}:=H\cap\text{Ad}^{-1}(\Gamma)$ is also solvable and of finite index in $H$. Taking the closure in $G$, we obtain  that $\overline{H_{\Gamma}}\subset G$ is a solvable subgroup of finite index. Since a closed subgroup of finite index is also open, it follows that $\overline{H_{\Gamma}}=G$ since $G$ is connected. In particular, $G$ is solvable. This contradicts that the Lie algebra $\mathfrak{g}$ is perfect by Prop.~\ref{prop:obstruction}.
	\end{remark}
	
	
	
	\section{Creating new infinitesimally rigid Lie foliations out of old ones}\label{sec:four}
	
	A Lie $\mathfrak{g}$-foliation $\mathcal{F}$ can always be enlarged to a Lie $\mathfrak{g}/\mathfrak{h}$-foliation $\mathcal{G}$ whenever $\mathfrak{h}\subset\mathfrak{g}$ is an ideal. We now show that, if $\mathcal{F}$ is infinitesimally rigid with dense leaves, then the same holds for its enlargement $\mathcal{G}$ provided that the ideal $\mathfrak{h}$ is perfect. Consequently, if $\mathcal{F}$ is modeled on a semisimple Lie algebra $\mathfrak{g}$, then dividing out simple ideals repeatedly yields an infinitesimally rigid Lie foliation with dense leaves modeled on a simple Lie algebra.
	
	\vspace{0.3cm}
	
	We will make use of the Serre spectral sequence for Lie algebroid cohomology recently developed in \cite{Serre}. If $\mathcal{F}$ and $\mathcal{G}$ are two foliations such that the leaves of $\mathcal{F}$ are contained in the leaves of $\mathcal{G}$, then $T\mathcal{F}$ is a wide Lie subalgebroid of $T\mathcal{G}$. This yields a spectral sequence $E_k$ converging to the foliated cohomology $H^{\bullet}(\mathcal{G})$, see \cite[Thm.~1.1]{Serre}.
	To describe it explicitly, note that $T\mathcal{F}$ has a canonical representation $\nabla$ on the normal bundle $T\mathcal{G}/T\mathcal{F}$ given by
	\[
	\nabla_{X}\overline{Y}:=\overline{[X,Y]},\hspace{1cm}X\in\Gamma(T\mathcal{F}),\overline{Y}\in\Gamma(T\mathcal{G}/T\mathcal{F}).
	\]
	In the formula above, we denoted by $\overline{Y}\in\Gamma(T\mathcal{G}/T\mathcal{F})$ the image of $Y\in\Gamma(T\mathcal{G})$ under the projection $T\mathcal{G}\rightarrow T\mathcal{G}/T\mathcal{F}$.
	Recall that $\nabla$ induces a dual representation $\nabla^{*}$ on $(T\mathcal{G}/T\mathcal{F})^{*}$, which is related to $\nabla$ by the Leibniz rule
	\[
	X\langle \xi,\overline{Y}\rangle = \langle \nabla_{X}^{*}\xi,\overline{Y}\rangle + \langle\xi,\nabla_{X}\overline{Y}\rangle,\hspace{0.5cm}\forall X\in\Gamma(T\mathcal{F}),\xi\in\Gamma(T\mathcal{G}/T\mathcal{F})^{*}, \overline{Y}\in\Gamma(T\mathcal{G}/T\mathcal{F}).
	\]
	We extend $\nabla^{*}$ to a $T\mathcal{F}$-representation on $\wedge^{p}(T\mathcal{G}/T\mathcal{F})^{*}$. The zeroth page of the spectral sequence $E_k$ is then given by
	\[
	E_0^{p,q}\cong \Omega^{q}\big(\mathcal{F},\wedge^{p}(T\mathcal{G}/T\mathcal{F})^{*}\big),
	\]
	and its differential $d_0$ coincides with the differential $d_{\nabla^{*}}$ induced by the representation $\nabla^{*}$, see \cite[Thm.~3.6]{Serre}. It follows that
	\begin{equation}\label{eq:E_1}
		E_1^{p,q}\cong H^{q}\big(\mathcal{F},\wedge^{p}(T\mathcal{G}/T\mathcal{F})^{*}\big).
	\end{equation}
	In particular, we have
	\begin{equation}\label{eq:E_1p}
		E_1^{p,0}\cong\big\{\xi\in\Gamma\big(\wedge^{p}(T\mathcal{G}/T\mathcal{F})^{*}\big):\nabla_X^{*}\xi=0\ \ \forall X\in\Gamma(T\mathcal{F})\big\}.
	\end{equation}
	To express this in more familiar terms, let us introduce the subcomplex $\Omega^{\bullet}_{\mathcal{F}}(\mathcal{G})\subset\big(\Omega^{\bullet}(\mathcal{G}),d_{\mathcal{G}}\big)$ consisting of $\mathcal{F}$-basic foliated forms on $\mathcal{G}$, i.e.
	\begin{equation}\label{eq:Omega_bas}
		\Omega^{p}_{\mathcal{F}}(\mathcal{G})=\left\{\xi\in\Omega^{p}(\mathcal{G}):\iota_X\xi=\pounds_X \xi=0\ \ \forall X\in\Gamma(T\mathcal{F})\right\}.
	\end{equation}
	Now note that for all $X\in\Gamma(T\mathcal{F}),\xi\in\Gamma(T\mathcal{G}/T\mathcal{F})^{*}$ and $\overline{Y}\in\Gamma(T\mathcal{G}/T\mathcal{F})$, we have
	\begin{align*}
		\langle \nabla_{X}^{*}\xi,\overline{Y}\rangle&=X\langle \xi,\overline{Y}\rangle-\langle\xi,\nabla_{X}\overline{Y}\rangle\\
		&=X\big(\xi(Y)\big)-\xi([X,Y])\\
		&=\big(\pounds_{X}\xi\big)(Y).
	\end{align*}
	This shows that $\nabla_{X}^{*}$ and $\pounds_X$ coincide on sections of 
	$(T\mathcal{G}/T\mathcal{F})^{*}$, which we identify with foliated forms on $\mathcal{G}$ annihilating $T\mathcal{F}$. Since $\nabla_{X}^{*}$ and $\pounds_X$ satisfy the same derivation rule with respect to the wedge product, they also coincide as operators on $\Gamma\big(\wedge^{p}(T\mathcal{G}/T\mathcal{F})^{*}\big)$. Comparing \eqref{eq:E_1p} and \eqref{eq:Omega_bas}, it is now clear that
	\[
	E_1^{p,0}\cong\Omega^{p}_{\mathcal{F}}(\mathcal{G}).
	\]
	Under this identification, the differential $d_1:E_1^{\bullet,0}\rightarrow E_1^{\bullet+1,0}$ of the first page is just the restriction of $d_{\mathcal{G}}$ to $\Omega^{\bullet}_{\mathcal{F}}(\mathcal{G})$, and therefore
	\begin{equation}\label{eq:bas-coho}
		E_2^{p,0}\cong H^{p}_{\mathcal{F}}(\mathcal{G}).
	\end{equation}
	
	\begin{prop}\label{prop:spectral}
		Let $\mathcal{F}$ be a Lie $\mathfrak{g}$-foliation on a manifold $M$ defined by a Maurer-Cartan form $\omega\in\Omega^{1}(M,\mathfrak{g})$. Let $\mathfrak{h}\subset\mathfrak{g}$ be an ideal and denote by $pr:\mathfrak{g}\rightarrow \mathfrak{g}/\mathfrak{h}$ the projection. 
		\begin{enumerate}
			\item The form $pr\circ\omega\in\Omega^{1}(M,\mathfrak{g}/\mathfrak{h})$ defines a Lie $\mathfrak{g}/\mathfrak{h}$-foliation $\mathcal{G}$.
			\item Assume that $\mathcal{F}$ is infinitesimally rigid with dense leaves. Then $H^{1}(\mathcal{G})\cong H^{1}(\mathfrak{h})$.
		\end{enumerate}
	\end{prop}
	\begin{proof}
		$(1)$ For every point $x\in M$ we have a surjective map
		\[
		pr\circ\omega_x:T_{x}M\twoheadrightarrow\mathfrak{g}\twoheadrightarrow\mathfrak{g}/\mathfrak{h}.
		\]
		Moreover, $pr\circ\omega\in\Omega^{1}(M,\mathfrak{g}/\mathfrak{h})$ satisfies the Maurer-Cartan equation because the projection $pr:\mathfrak{g}\rightarrow \mathfrak{g}/\mathfrak{h}$ is a Lie algebra homomorphism, i.e.
		\[
		0=pr\circ\left(d\omega+\frac{1}{2}[\omega,\omega]\right)=d\big(pr\circ\omega\big)+\frac{1}{2}[pr\circ\omega,pr\circ\omega].
		\]
		$(2)$ We start by examining the Serre spectral sequence $E_k$ of the Lie subalgebroid $T\mathcal{F}\subset T\mathcal{G}$. Since it converges to $H^{\bullet}(\mathcal{G})$, we have that
		\[
		H^{1}(\mathcal{G})\cong E_{\infty}^{1,0}\oplus E_{\infty}^{0,1}.
		\]
		By the isomorphism \eqref{eq:E_1}, we have that $E_1^{0,1}\cong H^{1}(\mathcal{F})$, which vanishes since $\mathcal{F}$ is assumed to be infinitesimally rigid. Hence, also the limiting term $E_{\infty}^{0,1}$ vanishes. Because the differential $d_k$ of page $E_k$ has bi-degree $(k,1-k)$, it follows that
		\[
		E_{\infty}^{1,0}=E_2^{1,0}\cong H^{1}_{\mathcal{F}}(\mathcal{G}),
		\]
		where we used the isomorphism \eqref{eq:bas-coho}. So the result follows if we show that $H^{1}_{\mathcal{F}}(\mathcal{G})\cong H^{1}(\mathfrak{h})$.
		
		We will show that the complex  $\big(\Omega^{\bullet}_{\mathcal{F}}(\mathcal{G}),d_{\mathcal{G}}\big)$ is isomorphic with the Chevalley-Eilenberg complex $\big(\wedge^{\bullet}\mathfrak{h}^{*},d_{CE}\big)$.
		To do so, recall that the space of transverse fields $\mathfrak{X}(M,\mathcal{F})/\Gamma(T\mathcal{F})$ is canonically a Lie algebra, where $\mathfrak{X}(M,\mathcal{F})$ is the space of $\mathcal{F}$-projectable vector fields. Since the leaves of $\mathcal{F}$ are dense, we have a natural Lie algebra isomorphism
		\[
		\Psi:\big(\mathfrak{g},[\cdot,\cdot]\big)\overset{\sim}{\longrightarrow}\left(\mathfrak{X}(M,\mathcal{F})/\Gamma(T\mathcal{F}),[\cdot,\cdot]\right):v\mapsto \overline{Y_v},
		\]
		where $\overline{Y_v}$ is the unique transverse field satisfying $\omega(\overline{Y_v})=v$.
		Fix a basis $\{v_1,\ldots,v_k\}$ of $\mathfrak{h}$ and let $\overline{Y_1},\ldots,\overline{Y_k}\in\mathfrak{X}(M,\mathcal{F})/\Gamma(T\mathcal{F})$ denote the corresponding transverse fields.
		It follows that the space $\mathbb{R}\overline{Y_1}\oplus\cdots\oplus\mathbb{R}\overline{Y_k}$ inherits a Lie algebra structure and that we have an isomorphism of Lie algebras
		\begin{equation}\label{eq:iso-Lie}
			\Psi:\big(\mathfrak{h},[\cdot,\cdot]\big)\overset{\sim}{\longrightarrow}\left(\mathbb{R}\overline{Y_1}\oplus\cdots\oplus\mathbb{R}\overline{Y_k},[\cdot,\cdot]\right).
		\end{equation}
		Note that any representative $Y_i$ of the class $\overline{Y_i}$ is in fact tangent to $\mathcal{G}$ because
		\[
		(pr\circ\omega)(Y_i)=pr(v_i)=0.
		\]
		It follows that $\{\overline{Y_1},\ldots,\overline{Y_k}\}$ is a frame of $T\mathcal{G}/T\mathcal{F}$. Moreover, each  $\overline{Y_i}$ is flat with respect to the connection $\nabla$ because any representative $Y_i$ of $\overline{Y_i}$ is $\mathcal{F}$-projectable. This implies that the dual frame $\{\xi_1,\ldots,\xi_k\}$ of $(T\mathcal{G}/T\mathcal{F})^{*}$ consists of sections that are flat for $\nabla^{*}$. Because the leaves of $\mathcal{F}$ are dense, the description \eqref{eq:E_1p} of $\Omega^{p}_{\mathcal{F}}(\mathcal{G})$ shows that
		\[
		\Omega^{p}_{\mathcal{F}}(\mathcal{G})\cong\bigoplus_{1\leq i_1<\cdots<i_p\leq k}\mathbb{R}~\xi_{i_1}\wedge\cdots\wedge\xi_{i_p}.
		\]
		Since 
		\[
		d_{\mathcal{G}}\xi_\alpha\big(\overline{Y_\beta},\overline{Y_\gamma}\big)=-\xi_{\alpha}[\overline{Y_\beta},\overline{Y_\gamma}],
		\]
		we obtain that the complex $\big(\Omega^{\bullet}_{\mathcal{F}}(\mathcal{G}),d_{\mathcal{G}}\big)$ is nothing but the Chevalley-Eilenberg complex of the Lie algebra $\left(\mathbb{R}\overline{Y_1}\oplus\cdots\oplus\mathbb{R}\overline{Y_k},[\cdot,\cdot]\right)$. Since the map \eqref{eq:iso-Lie} is a Lie algebra isomorphism, exterior powers of its dual give an isomorphism of Chevalley-Eilenberg complexes
		\[
		\wedge^{\bullet}\Psi^{*}:\big(\Omega^{\bullet}_{\mathcal{F}}(\mathcal{G}),d_{\mathcal{G}}\big)\overset{\sim}{\longrightarrow}\big(\wedge^{\bullet}\mathfrak{h}^{*},d_{CE}\big).
		\]
		This implies in particular that $H^{1}_{\mathcal{F}}(\mathcal{G})\cong H^{1}(\mathfrak{h})$, so the proof is finished.
	\end{proof}
	
	\begin{cor}\label{cor:semi-simple}
		Let $\mathcal{F}$ be an infinitesimally rigid Lie $\mathfrak{g}$-foliation with dense leaves. 
		\begin{enumerate}
			\item If $\mathfrak{h}\subset\mathfrak{g}$ is a perfect ideal, then $\mathcal{F}$ is contained in a Lie $\mathfrak{g}/\mathfrak{h}$-foliation $\mathcal{G}$ with dense leaves that is still infinitesimally rigid.
			\item In particular, assume that $\mathfrak{g}$ is semisimple with decomposition into simple ideals given by $\mathfrak{g}=\mathfrak{g}_1\oplus\cdots\oplus\mathfrak{g}_n$. Then there exist infinitesimally rigid Lie $\mathfrak{g}_{j}$-foliations $\mathcal{G}_j$ with dense leaves such that $\mathcal{F}\subset\mathcal{G}_{j}$ for $j=1,\ldots,n$.
		\end{enumerate}
	\end{cor}
	
	\section{Constructing infinitesimally rigid Lie foliations with dense leaves}\label{sec:five}
	This section is devoted to our main result. Using the suspension method, we construct infinitesimally rigid Lie foliations with dense leaves on compact manifolds. Our construction relies on two auxiliary results. First, we use that the deformation cohomology of a suspension foliation can be expressed in terms of the group cohomology for a suitable action of a discrete group \cite{Kacimi}. Next, we will use a vanishing result for group cohomology which applies when the discrete group acting has Kazhdan's property (T) \cite{Zimmer}. Our construction yields examples of infinitesimally rigid Lie $\mathfrak{g}$-foliations with dense leaves on compact manifolds for any compact, semisimple Lie algebra $\mathfrak{g}$ whose simple factors are different from $\mathfrak{so}(3)$.

	\subsection{Suspension foliations}\label{sec:suspension}
	The examples of infinitesimally rigid Lie foliations that we will construct are obtained by the suspension method, which we now briefly recall. Let $B$ and $T$ be compact, connected manifolds and assume that we are given a group homomorphism
	\[
	\varphi:\pi_{1}(B)\rightarrow\text{Diff}(T),
	\]
	where $\text{Diff}(T)$ denotes the diffeomorphism group of $T$. Let $\widetilde{B}$ be the universal cover of $B$, and consider the product manifold $\widetilde{B}\times T$. We get an action
	$
	\pi_1(B)\curvearrowright \widetilde{B}\times T
	$
	defined by
	\begin{equation}\label{eq:action}
		\gamma\cdot(\tilde{b},t)=\big(\gamma\star\tilde{b},\varphi(\gamma)(t)\big),
	\end{equation}
	where $\star$ denotes the natural action of $\pi_1(B)$ on $\widetilde{B}$. The action $\pi_1(B)\curvearrowright \widetilde{B}\times T$ is free and properly discontinuous, hence the quotient $B_{\varphi}:=(\widetilde{B}\times T)/\pi_1(B)$ is smooth and the quotient map is a covering map
	\[
	\pi:\widetilde{B}\times T\rightarrow B_{\varphi}.
	\] 
	The quotient manifold $B_{\varphi}$ is a fiber bundle over $B$ with fiber $T$. The foliation by horizontal slices $\widetilde{B}\times \{t\}$ on $\widetilde{B}\times T$ is invariant under the action \eqref{eq:action}, hence it descends to a foliation $\mathcal{F}$ on $B_{\varphi}$ transverse to the fibers of $B_{\varphi}\rightarrow B$. 
	
	\begin{defi}
		The foliated manifold $\big(B_{\varphi},\mathcal{F}\big)$ is the suspension of $\varphi:\pi_{1}(B)\rightarrow\text{Diff}(T)$.
	\end{defi}
	
	We will be interested in Lie foliations of suspension type -- such foliations are in particular Riemannian. The suspension of a homomorphism $\varphi:\pi_{1}(B)\rightarrow\text{Diff}(T)$ gives a Riemannian foliation exactly when $\varphi$ takes values in the isometry group $\text{Isom}(T,g)\subset\text{Diff}(T)$ of $(T,g)$ for some choice of Riemannian metric $g$, see \cite[\S~3.7]{Molino}. Since our aim is to construct Lie foliations with dense leaves, it will be useful to study the leaves of a suspension foliation in detail. We first introduce some notation which is used in Lemma \ref{lem:suspension} below. Assume that $\varphi:\pi_{1}(B)\rightarrow\text{Isom}(T,g)$ is a homomorphism. For every $t\in T$, we denote by $\text{Orb}(t)\subset T$ the $\pi_{1}(B)$-orbit through $t$ and by $\text{Stab}(t)\subset\pi_{1}(B)$ the stabilizer of $t$. We also let $K$ denote the closure of $\varphi\big(\pi_{1}(B)\big)$ in the compact Lie group $\text{Isom}(T,g)$.
	
	\begin{lemma}\label{lem:suspension}
		Let $\mathcal{F}$ be the suspension foliation of a homomorphism $\varphi:\pi_{1}(B)\rightarrow\text{Isom}(T,g)$. 
		\begin{enumerate}
			\item The leaves of $\mathcal{F}$ are of the form
			\[
			(\widetilde{B}\times \text{Orb}(t))/\pi_1(B)\cong \widetilde{B}/\text{Stab}(t)\hspace{0.5cm}\text{for}\ t\in T.
			\]
			\item Compact leaves of $\mathcal{F}$ correspond with finite $\pi_{1}(B)$-orbits in $T$.
			\item If the action $K\curvearrowright T$ is transitive, then all leaves of $\mathcal{F}$ are dense. 
		\end{enumerate}
	\end{lemma}
	\begin{proof}
		$(1)$ We have a smooth surjection
		\[
		f:\widetilde{B}\times \text{Orb}(t)\rightarrow \widetilde{B}/\text{Stab}(t):\big(\tilde{b},\varphi(\gamma)(t)\big)\mapsto [\gamma^{-1}\star\tilde{b}].
		\]
		This map is well-defined, for if $\varphi(\gamma_1)(t)=\varphi(\gamma_2)(t)$ then $\gamma_2^{-1}\gamma_1\in\text{Stab}(t)$ hence 
		$$
		[\gamma_1^{-1}\star\tilde{b}]=[\gamma_2^{-1}\gamma_1\star(\gamma^{-1}\star\tilde{b})]=[\gamma_2^{-1}\star\tilde{b}].
		$$
		The map $f$ descends to the quotient $(\widetilde{B}\times \text{Orb}(t))/\pi_1(B)$ since for $\eta\in\pi_1(B)$ we have
		\[
		f\left(\eta\cdot\big(\tilde{b},\varphi(\gamma)(t)\big)\right)=f\left(\eta\star\tilde{b},\varphi(\eta\gamma)(t)\right)=[\gamma^{-1}\eta^{-1}\star(\eta\star\tilde{b})]=[\gamma^{-1}\star\tilde{b}]=f\big(\tilde{b},\varphi(\gamma)(t)\big).
		\]
		The resulting map $\overline{f}:(\widetilde{B}\times \text{Orb}(t))/\pi_1(B)\rightarrow\widetilde{B}/\text{Stab}(t)$ is injective, for if $\big(\tilde{b}_1,\varphi(\gamma_1)(t)\big)$ and $\big(\tilde{b}_2,\varphi(\gamma_2)(t)\big)$ have the same image under $f$, then there exists $\eta\in\text{Stab}(t)$ such that
		\[
		\eta\star\gamma_1^{-1}\star\tilde{b}_1=\gamma_2^{-1}\star\tilde{b}_2
		\]
		and therefore
		\[
		\big[\big(\tilde{b}_1,\varphi(\gamma_1)(t)\big)\big]=\big[\big(\gamma_1^{-1}\star\tilde{b}_1,t\big)\big]=\big[\big(\eta\star\gamma_1^{-1}\star\tilde{b}_1,t\big)\big]=\big[\big(\gamma_2^{-1}\star\tilde{b}_2,t\big)\big]=\big[\big(\tilde{b}_2,\varphi(\gamma_2)(t)\big)\big].
		\]
		This way, we obtain a diffeomorphism
		\[
		\overline{f}:(\widetilde{B}\times \text{Orb}(t))/\pi_1(B)\rightarrow\widetilde{B}/\text{Stab}(t):\big[\big(\tilde{b},\varphi(\gamma)(t)\big)\big]\mapsto[\gamma^{-1}\star\tilde{b}]
		\]
		with inverse
		\[
		\overline{f}^{-1}:\widetilde{B}/\text{Stab}(t)\rightarrow(\widetilde{B}\times \text{Orb}(t))/\pi_1(B):\big[\tilde{b}\big]\mapsto\big[\big(\tilde{b},t\big)\big].
		\]
		
		$(2)$ For each $t\in T$, the leaf of $\mathcal{F}$ coming from $\widetilde{B}\times\{t\}$ is a covering space $p:\widetilde{B}/\text{Stab}(t)\rightarrow B$  with compact base $B$. This implies that $\widetilde{B}/\text{Stab}(t)$ is compact exactly when the cardinality of a fiber $p^{-1}(x_0)$ is finite. Since $\pi_1\big(\widetilde{B}/\text{Stab}(t)\big)=\text{Stab}(t)$, there is a bijection between the fiber $p^{-1}(x_0)$ and the coset space $\pi_1(B)/\text{Stab}(t)$. The orbit-stabilizer theorem gives a bijection between the latter and the orbit $\text{Orb}(t)$. In conclusion, the leaf of $\mathcal{F}$ coming from $\widetilde{B}\times\{t\}$ is compact exactly when the orbit $\text{Orb}(t)$ is finite.

		$(3)$ The closure of the leaf of $\mathcal{F}$ coming from $\widetilde{B}\times\{t\}$ is given by 
		\begin{equation}\label{eq:leaf-closure}
			\left(\widetilde{B}\times K\cdot \{t\}\right)/\pi_1(B),
		\end{equation}
		see \cite[\S~3.7]{Molino}. Since the action $K\curvearrowright T$ is transitive, its orbit $K\cdot \{t\}$ is all of $T$ and therefore the leaf closure \eqref{eq:leaf-closure} is all of $M$. This shows that the leaves of $\mathcal{F}$ are dense.
	\end{proof}



	
	We will consider the particular case in which the manifold $T$ is a compact Lie group $G$. Endowing $G$ with a left invariant metric $g$, the group $G$ embeds into the isometry group $\text{Isom}(G,g)$ as the subgroup of left translations. Now take a homomorphism
	$
	\varphi:\pi_1(B)\rightarrow G,
	$
	such that the image $\varphi\big(\pi_1(B)\big)\subset G$ is a dense subgroup. It turns out that the associated suspension $(B_{\varphi},\mathcal{F})$ is a Lie foliation with dense leaves, as we now show.
	
	\begin{lemma}\label{lem:Lie-fol}
		Let $G$ be a compact, connected Lie group with Lie algebra $\mathfrak{g}$. Let $B$ be a compact, connected manifold and $\varphi:\pi_1(B)\rightarrow G$ a homomorphism with dense image. The suspension of $\varphi$ gives a Lie $\mathfrak{g}$-foliation $\mathcal{F}$ with dense leaves on the compact manifold $B_{\varphi}$.
	\end{lemma}
	\begin{proof}
		 We first show that $\mathcal{F}$ is a Lie $\mathfrak{g}$-foliation, meaning that there exists a non-singular Maurer-Cartan form $\omega\in\Omega^{1}(B_{\varphi},\mathfrak{g})$ such that $T\mathcal{F}=\ker\omega$. We start by defining a $\mathfrak{g}$-valued one-form $\widetilde{\omega}$ on $\widetilde{B}\times G$ by setting
		\[
		\widetilde{\omega}_{(\tilde{b},g)}:T_{\tilde{b}}\widetilde{B}\times T_{g}G\rightarrow\mathfrak{g}:(v,w)\mapsto \big(L_{g^{-1}}\big)_{*}w.
		\]
		Here $L_{g}$ denotes the left multiplication by $g\in G$. Note that $\widetilde{\omega}$ is just the canonical flat connection on the trivial principal bundle $\widetilde{B}\times G$, see \cite[Ex.~4.13]{Moerdijk}. It follows that $\widetilde{\omega}$ satisfies the Maurer-Cartan equation \eqref{eq:MC}. Moreover, $\widetilde{\omega}$ is pointwise surjective and its kernel is the tangent distribution to the fibers of $\widetilde{B}\times G\rightarrow G$. The result follows if we can show that $\widetilde{\omega}$ descends to $B_{\varphi}$, i.e. that it is invariant under the action
		\[
		\psi_{\gamma}:\widetilde{B}\times G\rightarrow \widetilde{B}\times G:(\tilde{b},g)\mapsto \big(\gamma\star\tilde{b},\varphi(\gamma)g\big)
		\]
		for $\gamma\in\pi_1(B)$. To do so, we take $(v,w)\in T_{\tilde{b}}\widetilde{B}\times T_{g}G$ and compute
		\begin{align*}
			\left(\psi_{\gamma}^{*}\widetilde{\omega}\right)_{(\tilde{b},g)}(v,w)&=\widetilde{\omega}_{(\gamma\star\tilde{b},\varphi(\gamma)g)}\big((\psi_{\gamma})_{*}(v,w)\big)\\
			&=\big(L_{g^{-1}\varphi(\gamma)^{-1}}\big)_{*}\big((\text{pr}_{2})_{*}(\psi_{\gamma})_{*}(v,w)\big)\\
			&=\big(L_{g^{-1}\varphi(\gamma)^{-1}}\big)_{*}\big((L_{\varphi(\gamma)})_{*}(w)\big)\\
			&=\big(L_{g^{-1}}\big)_{*}w\\
			&=\widetilde{\omega}_{(\tilde{b},g)}(v,w).
		\end{align*}
		In the second equality, we denoted by $\text{pr}_{2}:\widetilde{B}\times G\rightarrow G$ the projection and we used the fact that $\text{pr}_{2}\circ\psi_{\gamma}=L_{\varphi(\gamma)}\circ\text{pr}_{2}$. This shows that $\widetilde{\omega}$ descends to $B_{\varphi}$, hence $\mathcal{F}$ is a Lie $\mathfrak{g}$-foliation.
		
		The fact that $\mathcal{F}$ has dense leaves follows from item $(3)$ of Lemma~\ref{lem:suspension}. Indeed, endowing $G$ with a left invariant metric $g$, we obtain a continuous map
		$
		L:G\rightarrow\text{Isom}(G,g):g\mapsto L_{g},
		$
		where $\text{Isom}(G,g)$ carries the compact-open topology. Denoting $H:=\varphi(\pi_1(B))\subset G$, item $(3)$ of Lemma~\ref{lem:suspension} asks us to check that $\overline{L(H)}\subset\text{Isom}(G,g)$ acts transitively on $G$. Continuity of the map $L$ implies that
		$
		L(G)=L(\overline{H})\subset \overline{L(H)}.
		$
		Since the action of $L(G)$ on $G$ is transitive, the same then holds for the action of $\overline{L(H)}$. It follows that the leaves of $\mathcal{F}$ are dense.
	\end{proof}
	
	
	In what follows, we will restrict our attention to Lie foliations $\mathcal{F}$ of the type described in Lemma~\ref{lem:Lie-fol}. We now proceed by looking for conditions under which the cohomology group $H^{1}(\mathcal{F},N\mathcal{F})$ vanishes. An important tool will be a recent result by El Kacimi Alaoui which describes $H^{1}(\mathcal{F},N\mathcal{F})$ for foliations $\mathcal{F}$ that are developable \cite{Kacimi}. We now recall this result.
	
	\subsection{Infinitesimal deformations of developable foliations}
	
	The cohomology group $H^{1}(\mathcal{F},N\mathcal{F})$ governing the deformations of a foliation $\mathcal{F}$ is  very hard to compute in general. The situation is better for developable foliations, thanks to work by El Kacimi Alaoui \cite{Kacimi}.
	
	\begin{defi}
		A foliation $\mathcal{F}$ on a connected manifold $M$ is developable if there exists a connected, normal covering $\pi:\widehat{M}\rightarrow M$ such that the pullback foliation $\widehat{\mathcal{F}}$ on $\widehat{M}$ is given by the fibers of a locally trivial fibration $D:\widehat{M}\rightarrow W$. 
	\end{defi}
	
	The deck transformation group $\Gamma$ of the covering $\pi:\widehat{M}\rightarrow M$ acts freely and properly discontinuously on $\widehat{M}$, and since the covering is normal we have that $\widehat{M}/\Gamma\cong M$. Moreover, since the action $\Gamma\curvearrowright\widehat{M}$ preserves the fibers of the developing map $D:\widehat{M}\rightarrow W$, there is an induced action $\Gamma\curvearrowright W$. In particular, the space of vector fields $\mathfrak{X}(W)$ is a $\Gamma$-module.
	
	\begin{ex}
		Given a homomorphism $\varphi:\pi_1(B)\rightarrow\text{Diff}(T)$, where $B$ and $T$ are connected compact manifolds, the suspension foliation $\mathcal{F}$ on $B_{\varphi}$ is developable. Indeed, we have a connected, normal covering 
		$
		\pi:\widetilde{B}\times T\rightarrow B_{\varphi},
		$
		and the pullback foliation $\widehat{\mathcal{F}}$ is given by the fibers of the second projection $\widetilde{B}\times T\rightarrow T$. The deck transformation group is $\Gamma:=\pi_1(B)$.
	\end{ex}
	
	For suitable developable foliations $\mathcal{F}$, one can express $H^{1}(\mathcal{F},N\mathcal{F})$ in terms of the group cohomology of the action $\Gamma\curvearrowright\mathfrak{X}(W)$. We first recall its underlying differential complex.
	
	Let $\Gamma$ be a discrete group acting on a vector space $V$. For each integer $n\geq 1$, denote by $C^{n}(\Gamma,V)$ the vector space of maps $\Gamma^{n}\rightarrow V$. By convention, $C^{0}(\Gamma,V)=V$. There is a differential $d$ on the graded vector space $C^{\bullet}(\Gamma,V)$, defined by
	\begin{align*}
		d\psi(\gamma_1,\ldots,\gamma_{n+1})&=\gamma_1\cdot\psi(\gamma_2,\ldots,\gamma_{n+1})\\
		&\hspace{0.5cm}+\sum_{i=1}^{n}(-1)^{i}\psi(\gamma_1,\ldots,\gamma_{i-1},\gamma_i\gamma_{i+1},\gamma_{i+2},\ldots,\gamma_{n+1})\\
		&\hspace{0.5cm}+(-1)^{n+1}\psi(\gamma_1,\ldots,\gamma_n).
	\end{align*}
	We denote the associated cohomology groups by $H^{\bullet}(\Gamma,V)$. Of particular interest to us is the first cohomology group $H^{1}(\Gamma,V)$, which is given by
	\begin{equation}\label{eq:group-coho}
		H^{1}(\Gamma,V)=\frac{\{\psi:\Gamma\rightarrow V:\ \psi(\gamma_1\gamma_2)=\psi(\gamma_1)+\gamma_1\cdot\psi(\gamma_2)\}}{\{\chi:\Gamma\rightarrow V:\ \chi(\gamma)=\gamma\cdot v-v\ \ \text{for some}\ v\in V\}}.
	\end{equation}
	In other words, $H^{1}(\Gamma,V)$ is the space of crossed homomorphisms $\Gamma\rightarrow V$ modulo those induced by elements of $V$. We can now state the following result from \cite{Kacimi}.
	
	\begin{prop}\label{prop:cohomology}
		Let $\mathcal{F}$ be a developable foliation on a connected manifold $M$ with normal covering $\pi:\widehat{M}\rightarrow M$, deck transformation group $\Gamma$ and developing map $D:\widehat{M}\rightarrow W$. Assume that the fiber $\widehat{L}$ of $D$ satisfies $H^{1}(\widehat{L})=0$. Then $H^{1}(\mathcal{F},N\mathcal{F})\cong H^{1}(\Gamma,\mathfrak{X}(W))$.
	\end{prop}
	
	\begin{ex}\label{ex:coh-dev}
		Let $\varphi:\pi_1(B)\rightarrow\text{Diff}(T)$ be a group homomorphism, where $B$ and $T$ are compact connected manifolds. The suspension foliation $\mathcal{F}$ on $B_{\varphi}$ lifts under $\pi:\widetilde{B}\times T\rightarrow B_{\varphi}$ to the foliation by fibers of $\widetilde{B}\times T\rightarrow T$. Since the universal cover $\widetilde{B}$ is simply connected, we have $H^{1}(\widetilde{B})=0$. Hence, Prop.~\ref{prop:cohomology} implies that 
		$
		H^{1}(\mathcal{F},N\mathcal{F})\cong H^{1}\big(\pi_1(B),\mathfrak{X}(T)\big).
		$
	\end{ex}
	
	We are particularly interested in the Lie foliations described in Lemma~\ref{lem:Lie-fol}. Let us specialize the conclusion of Prop.~\ref{prop:cohomology} to this class of foliations.
	
	\begin{cor}\label{cor:group-coho}
		Let $G$ be a compact, connected Lie group with Lie algebra $\mathfrak{g}$. Let $B$ be a compact, connected manifold and $\varphi:\pi_1(B)\rightarrow G$ any homomorphism. The deformation cohomology of the suspension foliation $\mathcal{F}$ of $\varphi$ is given by 
		\begin{equation}\label{eq:isomorphisms}
			H^{1}(\mathcal{F},N\mathcal{F})\cong H^{1}\big(\pi_1(B),\mathfrak{X}(G)\big)\cong H^{1}\big(\pi_1(B),C^{\infty}(G)\big)\otimes\mathfrak{g}.
		\end{equation}
	\end{cor}
	\begin{proof}
		The first isomorphism in \eqref{eq:isomorphisms} follows from Ex.~\ref{ex:coh-dev} above. For the second isomorphism in \eqref{eq:isomorphisms}, recall that $\mathfrak{X}(G)\cong C^{\infty}(G)\otimes\mathfrak{g}$. Since the action of $\gamma\in\pi_{1}(B)$ on $G$ is left translation by the element $\varphi(\gamma)\in G$, it is clear that the induced action $\pi_1(B)\curvearrowright\mathfrak{X}(G)$ is trivial on left invariant vector fields, i.e. elements of $\mathfrak{g}$. This implies the second isomorphism in \eqref{eq:isomorphisms}.
	\end{proof}
	
	Consequently, for Lie foliations of the type described in Lemma~\ref{lem:Lie-fol} to be infinitesimally rigid, it remains to impose conditions on the group action $\pi_1(B)\curvearrowright C^{\infty}(G)$ ensuring that the group cohomology $H^{1}\big(\pi_1(B),C^{\infty}(G)\big)$ vanishes. This is the aim of the next subsection.

	\subsection{Actions of Kazhdan groups}
	We will now impose that the fundamental group $\pi_1(B)$ has Kazhdan's property (T) as a discrete group, allowing us to invoke a vanishing result for the cohomology of such groups due to Lubotzky-Zimmer \cite{Zimmer}.
	For an extensive treatment of property (T) groups, we refer to the book \cite{Kazhdan}. Let us just recall the definition here.
	
	\begin{defi}
		Let $G$ be a topological group and $\mathcal{H}$ a complex Hilbert space. Assume that $\rho:G\rightarrow U(\mathcal{H})$ is a unitary representation.
		For a subset $Q\subset G$ and a constant $\epsilon>0$, a vector $\xi\in\mathcal{H}$ is $(Q,\epsilon)$-invariant if
		\[
		\sup_{g\in Q}\|\rho(g)\xi-\xi\|<\epsilon\|\xi\|.
		\]
	\end{defi}

	\begin{defi}
		Let $G$ be a topological group. 
		\begin{enumerate}[i)]
			\item A subset $Q\subset G$ is a Kazhdan set if there exists $\epsilon>0$ with the following property: every unitary representation $(\rho,\mathcal{H})$ of $G$ which has a $(Q,\epsilon)$-invariant vector also has a non-zero invariant vector. 
			\item The group $G$ has Kazhdan's property (T) if $G$ has a compact Kazhdan set. In this case, we also say that $G$ is a Kazhdan group.
		\end{enumerate}
	\end{defi}
	
	Compact topological groups have property (T) \cite[Prop.~1.1.5]{Kazhdan}, while $\mathbb{R}^{n}$ and $\mathbb{Z}^{n}$ do not \cite[Ex.~1.1.7]{Kazhdan}. Of special interest to us are semisimple Lie groups $G$ with property (T). A connected semisimple Lie group $G$ with Lie algebra $\mathfrak{g}$ has property (T) if and only if no simple factor of $\mathfrak{g}$ is isomorphic to $\mathfrak{so}(n,1)$ or $\mathfrak{su}(n,1)$ \cite[Thm.~3.5.4]{Kazhdan}.
	
	\vspace{0.3cm}
	
	Kazhdan's property (T) is relevant for our purpose because of the following: a countable discrete group $\Gamma$ has property (T) exactly when the first group cohomology $H^{1}(\Gamma,V)$ vanishes for every orthogonal representation of $\Gamma$ on a real Hilbert space $V$  \cite[Chapter 2]{Kazhdan}. This statement is not directly applicable in our situation because the representation space $C^{\infty}(G)$ that we encountered in Cor.~\ref{cor:group-coho} is not a Hilbert space. In more detail, the action $\pi_1(B)\curvearrowright C^{\infty}(G)$ is orthogonal if we endow $C^{\infty}(G)$ with the $L^{2}$-inner product coming from a left invariant metric on $G$. Therefore, imposing that $\pi_1(B)$ has Kazhdan's property (T) a priori only yields the vanishing of $H^{1}(\pi_1(B),L^{2}(G))$, where $L^{2}(G)$ is the Hilbert space of square integrable functions on $G$.
	To overcome this issue, we will invoke a regularity result by Lubotzky-Zimmer \cite[Thm.~4.1]{Zimmer}. We only state a particular case of their result, which holds under more general assumptions that are implied by property (T). 
	
	\begin{prop}\label{prop:zimmer}
		Let $\Gamma$ be a finitely generated discrete group acting smoothly on a compact manifold $M$. Suppose that $\Gamma$ has Kazhdan's property (T), and that the $\Gamma$-action is isometric and ergodic. Then both $H^{1}(\Gamma,\mathfrak{X}(M))$ and $H^{1}(\Gamma,C^{\infty}(M))$ vanish.
	\end{prop}
	
	In the above, say that the $\Gamma$-action is isometric with respect to the Riemannian metric $g$ on $M$. Denoting by $\mu$ the Riemannian measure coming from $g$, it follows that $\Gamma$ acts by measure preserving automorphisms of $(M,\mu)$. By rescaling the metric $g$, we can ensure that $\mu$ is a probability measure. The action $\Gamma\curvearrowright(M,\mu)$ is called ergodic if for every $\Gamma$-invariant measurable subset $N\subset M$, one either has $\mu(N)=0$ or $\mu(N)=1$. Recall that $N\subset M$ is $\Gamma$-invariant if $\mu\big((\gamma\cdot N)\Delta N\big)=0$ for all $\gamma\in\Gamma$, where $\Delta$ denotes the symmetric difference. 
	
	\subsection{The main result}
	Now all preliminaries are in place to prove the main result. It yields infinitesimally rigid Lie foliations with dense leaves, of the type described in Lemma~\ref{lem:Lie-fol}.
	
	\begin{thm}\label{thm:main-result}
		Let $B$ be a connected, compact manifold such that $\pi_1(B)$ has property (T) as a discrete group. Let $G$ be a connected, compact Lie group and assume that $\varphi:\pi_1(B)\rightarrow G$ is a group homomorphism such that $\varphi(\pi_1(B))\subset G$ is dense. Then the suspension of $\varphi$ gives a Lie $\mathfrak{g}$-foliation $\mathcal{F}$ with dense leaves on a compact manifold, such that $H^{1}(\mathcal{F},N\mathcal{F})=0$.
	\end{thm}
	\begin{proof}
		We already showed in Lemma~\ref{lem:Lie-fol} that $\mathcal{F}$ is a Lie $\mathfrak{g}$-foliation with dense leaves. By Cor.~\ref{cor:group-coho}, we know that
		\begin{equation}\label{eq:must-vanish}
			H^{1}(\mathcal{F},N\mathcal{F})\cong H^{1}\big(\pi_1(B),\mathfrak{X}(G)\big)\cong H^{1}\big(\pi_1(B),C^{\infty}(G)\big)\otimes\mathfrak{g}.
		\end{equation}
		We now check that the assumptions of Prop.~\ref{prop:zimmer} are satisfied. Endowing $G$ with a left invariant metric, it is clear that the action of $\pi_1(B)$ on $G$ is isometric. The metric defines a left invariant volume form and therefore a Haar measure $\mu$ on $G$. By rescaling the metric, we can make sure that the Haar measure $\mu$ is normalized.
		It is well-known that if $\Gamma$ is a dense subgroup of a compact Lie group $G$, then the action by left translations $\Gamma\curvearrowright G$ is ergodic with respect to the normalized Haar measure $\mu$ on $G$ \cite[Lemma 2.4.1]{Popa}. Since by assumption $\varphi(\pi_1(B))\subset G$ is dense, it follows that the action $\pi_1(B)\curvearrowright (G,\mu)$ is ergodic. 
		At last, the fundamental group $\pi_1(B)$ of a compact manifold $B$ is finitely generated. So all assumptions of Prop.~\ref{prop:zimmer} are satisfied and therefore the cohomology group \eqref{eq:must-vanish} vanishes. 
	\end{proof}
	
	\begin{remark}
		We check by hand that the Lie foliations constructed in Thm.~\ref{thm:main-result} indeed satisfy the conclusion of Prop.~\ref{prop:obstruction} and Cor.~\ref{cor:not-amenable}. Below, we consider the suspension $(B_{\varphi},\mathcal{F})$ of a homomorphism $\varphi:\pi_1(B)\rightarrow G$, where $\pi_1(B)$ has Kazhdan's property (T) as a discrete group, $G$ is a connected compact Lie group and $\varphi(\pi_1(B))\subset G$ is a dense subgroup.
		
		\vspace{0.2cm}
		
		$1)$ We first check that the Lie algebra $\mathfrak{g}$ is perfect. To do so, note that also $\varphi(\pi_1(B))$ has property (T) as a discrete group, because of \cite[Thm.~1.3.4]{Kazhdan}. So $G$ has a dense, finitely generated subgroup $\varphi(\pi_1(B))$ with property (T) as a discrete group, hence $\overline{[G,G]}=G$ by \cite[Thm.~3]{Cornulier}. However, since $G$ is compact and connected, the commutator subgroup $[G,G]\subset G$ is a closed Lie subgroup with Lie algebra $[\mathfrak{g},\mathfrak{g}]$, see \cite[Thm.~5.21]{Sepanski}. So we have in fact that $[G,G]=G$, which at the infinitesimal level gives $[\mathfrak{g},\mathfrak{g}]=\mathfrak{g}$.

		$2)$ We now check that $H^{1}(B_{\varphi})$ vanishes. Recall that we have a covering $\widetilde{B}\times G\rightarrow B_{\varphi}$ with deck transformation group $\pi_1(B)$, see \S\ref{sec:suspension}. Its Cartan-Leray spectral sequence $(E_k,d_k)$ converges to $H^{\bullet}(B_{\varphi})$ and the $E_2$-term is given by \cite[Thm.~$8^{bis}.9$]{Spectral}
		\[
		E_2^{p,q}=H^{p}\left(\pi_1(B),H^{q}(\widetilde{B}\times G)\right).
		\]
		In particular, we get
		\[
		H^{1}(B_{\varphi})\cong E_{\infty}^{1,0}\oplus E_{\infty}^{0,1}.
		\]
		To study the term $E_{\infty}^{0,1}$, recall that
		$$
		E_{2}^{0,1}=H^{0}\left(\pi_1(B),H^{1}(G)\right)=H^{1}(G)^{\pi_1(B)},
		$$
		which is the space of invariants of the action $\pi_1(B)\curvearrowright H^{1}(G)$. Recall that the action of $\gamma\in\pi_1(B)$ on $G$ is left translation by the element $\varphi(\gamma)\in G$. Since $G$ is a compact and connected Lie group, every cohomology class in $H^{1}(G)$ can be represented by a left invariant closed one-form. It follows that the action $\pi_1(B)\curvearrowright H^{1}(G)$ is trivial, hence
		\[
		E_{2}^{0,1}=H^{1}(G)\cong H^{1}(\mathfrak{g})=0,
		\]
		where we used that $\mathfrak{g}$ is perfect. Hence, also the limiting term $E_{\infty}^{0,1}$ vanishes. As for the other limiting term $E_{\infty}^{1,0}$, the fact that the differential $d_k$ has bidegree $(k,1-k)$ implies that
		\[
		E_{\infty}^{1,0}=E_2^{1,0}=H^{1}\left(\pi_1(B),\mathbb{R}\right).
		\]
		Since the action $\pi_1(B)\curvearrowright \mathbb{R}$ is trivial, the first group cohomology $H^{1}\left(\pi_1(B),\mathbb{R}\right)$ reduces to the space of group homomorphisms $\pi_1(B)\rightarrow(\mathbb{R},+)$, see \eqref{eq:group-coho}. Such a homomorphism factors through the abelianization $\pi_1(B)/[\pi_1(B),\pi_1(B)]$, which is finite since  $\pi_1(B)$ has property (T) as a discrete group \cite[Cor.~1.3.6]{Kazhdan}. This implies that the only homomorphism $\pi_1(B)\rightarrow(\mathbb{R},+)$ is the zero map, showing that also $E_{\infty}^{1,0}$ vanishes. This confirms that
		\[
		H^{1}(B_{\varphi})\cong E_{\infty}^{1,0}\oplus E_{\infty}^{0,1}=0.
		\]
		
		$(3)$ We check that $\pi_1(B_{\varphi})$ is not amenable. By contradiction, assume that $\pi_1(B_{\varphi})$ were amenable. Consider the homotopy long exact sequence of the fiber bundle $G\rightarrow B_{\varphi}\rightarrow B$, a part of which looks like
		\[
		\cdots\rightarrow\pi_1(B_{\varphi})\rightarrow\pi_1(B)\rightarrow\pi_{0}(G).
		\]
		Since $G$ is connected, exactness implies that the morphism $\pi_1(B_{\varphi})\rightarrow\pi_1(B)$ is surjective. Since amenability is inherited by quotient groups, it follows that $\pi_1(B)$ is also amenable. However, discrete groups that are amenable and have property (T) are necessarily finite by \cite[Thm.~1.1.6]{Kazhdan}. Hence $\pi_1(B)$ is finite, which contradicts that $\varphi(\pi_1(B))\subset G$ is dense.
	\end{remark}
	
	\subsection{Existence of examples}
	It is not clear a priori whether Thm.~\ref{thm:main-result} admits any examples. In what follows, we show that there are lots of examples and we completely determine the Lie groups $G$ that may occur in Thm.~\ref{thm:main-result}. Note that because of Cor.~\ref{cor:semisimple}, we know that $G$ is necessarily semisimple. Since every finitely presented group is the fundamental group of some compact, connected manifold, we are led to answering the following question.
	
	\begin{question}
		Let $G$ be a connected, compact, semisimple Lie group. Is there a finitely presented discrete Kazhdan group $\Gamma$ admitting a homomorphism $\varphi:\Gamma\rightarrow G$ with dense image?
	\end{question}
	
	The answer is provided by the following result.
	
	\begin{thm}\label{thm:no-so3}
		Let $G$ be connected, compact and semisimple. The following are equivalent:
		\begin{enumerate}
			\item There is a finitely presented discrete Kazhdan group $\Gamma$ with a group homomorphism $\varphi:\Gamma\rightarrow G$ such that $\varphi(\Gamma)\subset G$ is dense.
			\item $SO(3)$ is not a quotient of $G$.
		\end{enumerate}
	\end{thm}

	The rest of this subsection is devoted to the proof of Thm.~\ref{thm:no-so3}. The main ingredient is a result by Margulis \cite[Chapter~III, Proof of Prop.~7]{Margulis}, who constructed dense subgroups with property (T) in any compact, connected, simple Lie group not locally isomorphic to $SO(3)$. More generally, de Cornulier determined in \cite{Cornulier} which connected Lie groups $G$ admit a finitely generated, dense subgroup that has Kazhdan's property (T) as a discrete group. It turns out that this question is in fact equivalent to the one we posed above.

	\begin{lemma}\label{lem:aux}
		For any Lie group $G$, the following are equivalent.
		\begin{enumerate}
			\item There is a finitely presented discrete Kazhdan group $\Gamma$ with a group homomorphism $\varphi:\Gamma\rightarrow G$ such that $\varphi(\Gamma)\subset G$ is dense.
			\item There is a finitely generated, dense subgroup $H\subset G$ with Kazhdan's property (T) as a discrete group.
		\end{enumerate}
	\end{lemma}
	\begin{proof}
		For $(1)\Rightarrow (2)$, we set $H:=\varphi(\Gamma)$. It is finitely generated, being a quotient of the finitely generated group $\Gamma$. It has property (T) as a discrete group by \cite[Thm.~1.3.4]{Kazhdan}. 
		
		For $(2)\Rightarrow (1)$, we proceed as follows. Since $H$ is a discrete group with property (T), we can take a finitely presented discrete group $\Gamma$ with property (T) admitting a normal subgroup $N\subset \Gamma$ such that $\Gamma/N\cong H$, see \cite[Thm.~3.4.5]{Kazhdan}. We then set $\varphi:\Gamma\rightarrow G$ to be the composition of the projection
		$
		\Gamma\twoheadrightarrow\Gamma/N\cong H
		$
		with the inclusion $H\hookrightarrow G$.
	\end{proof}
	
	
	
	\begin{proof}[Proof of Thm.~\ref{thm:no-so3}]
		First assume that $(1)$ holds, i.e. $\Gamma$ is a finitely presented discrete Kazhdan group and $\varphi:\Gamma\rightarrow G$ is a homomorphism with dense image. Assume by contradiction that $\psi:G\rightarrow SO(3)$ is a surjective Lie group homomorphism. Then $\psi\circ\varphi:\Gamma\rightarrow SO(3)$ is a homomorphism with dense image. This contradicts a result by Zimmer \cite[Thm.~7]{zimmer-actions} stating that any homomorphism from the discrete Kazhdan group $\Gamma$ to $SO(3)$ has finite image.
		
		Conversely, assume that $SO(3)$ is not a quotient of $G$. Denote by $p:\widetilde{G}\rightarrow G$ the simply connected covering of $G$. Then $\widetilde{G}=H_1\times\cdots\times H_k$ is a product of connected, compact, simple Lie groups $H_i$. None of them is locally isomorphic to $SO(3)$ because otherwise $SO(3)$ would be a quotient of $G$. The result \cite[Chapter~III, Proof of Prop.~7]{Margulis} by Margulis yields finitely generated, dense subgroups $\Gamma_i\subset H_i$ that have property (T) as a discrete group. It follows that $p(\Gamma_1\times\cdots\times\Gamma_k)\subset G$ is a finitely generated, dense subgroup with property (T) as a discrete group. By Lemma~\ref{lem:aux}, this implies that $(1)$ holds.
	\end{proof}
	\color{black}
	
	We would like to rephrase Thm.~\ref{thm:no-so3} in terms of the Lie algebra $\mathfrak{g}$, since Lie foliations are modeled on Lie algebras according to Def.~\ref{def:Lie}. To do so, we have the following result.
	
	\begin{lemma}\label{lem:rephrase}
		Let $G$ be a compact, connected, semisimple Lie group. Its Lie algebra $\mathfrak{g}$ decomposes into simple ideals $\mathfrak{g}=\mathfrak{g}_1\oplus\cdots\oplus\mathfrak{g}_k$. The following are equivalent:
		\begin{enumerate}
			\item There is no surjective Lie group homomorphism $G\rightarrow SO(3)$.
			\item No simple factor $\mathfrak{g}_{j}$ is isomorphic to $\mathfrak{so}(3)$.
		\end{enumerate}
	\end{lemma}
	\begin{proof}
		First assume that there is a surjective Lie group homomorphism $\varphi:G\rightarrow SO(3)$. Then there is an ideal $\mathfrak{h}\subset\mathfrak{g}$, coming from the closed normal subgroup $\ker\varphi\subset G$, such that $\mathfrak{g}/\mathfrak{h}\cong\mathfrak{so}(3)$. Necessarily, $\mathfrak{h}$ is of the form $\mathfrak{h}=\bigoplus_{i\in I}\mathfrak{g}_{i}$ for some subset $I\subset\{1,\ldots,k\}$. It then follows that one of the simple factors $\mathfrak{g}_i$ must be isomorphic to $\mathfrak{so}(3)$.
		
		Conversely, assume that $\mathfrak{so}(3)$ is a simple ideal in $\mathfrak{g}$. Take a complementary ideal $\mathfrak{p}\subset\mathfrak{g}$ such that $\mathfrak{g}=\mathfrak{p}\oplus\mathfrak{so}(3)$ as Lie algebras. If $P$ is the simply connected Lie group integrating $\mathfrak{p}$, then $\widetilde{G}:=P\times SU(2)$ is the simply connected Lie group integrating $\mathfrak{g}$. We get a surjective Lie group homomorphism
		$
		\varphi:\widetilde{G}\rightarrow SO(3),
		$
		which is the composition of the projection $\widetilde{G}\rightarrow SU(2)$ and the covering map $SU(2)\rightarrow SO(3)$. It remains to show that $\varphi$ descends to a Lie group homomorphism $\overline{\varphi}:G\rightarrow SO(3)$. Since $G$ is the quotient of $\widetilde{G}$ by a discrete central subgroup $\Gamma\subset Z\big(\widetilde{G}\big)$, it suffices to show that $\varphi\big(Z\big(\widetilde{G}\big)\big)=I$. This in turn will follow if we show that $\varphi\big(Z\big(\widetilde{G}\big)\big)\subset Z\big(SO(3)\big)$, since the center $Z\big(SO(3)\big)$ is $\{I\}$. To do so, take $\widetilde{a}\in Z\big(\widetilde{G}\big)$ and $b\in SO(3)$. Since $\varphi:\widetilde{G}\rightarrow SO(3)$ is surjective, we can take $\widetilde{b}\in\widetilde{G}$ such that $\varphi(\widetilde{b})=b$. It then follows that
		\[
		\varphi(\widetilde{a})b=\varphi(\widetilde{a}\widetilde{b})=\varphi(\widetilde{b}\widetilde{a})=b\varphi(\widetilde{a}),
		\]
		showing that $\varphi(\widetilde{a})\in Z\big(SO(3)\big)$. This confirms that $\varphi\big(Z\big(\widetilde{G}\big)\big)\subset Z\big(SO(3)\big)$. So $\varphi$ induces a surjective Lie group homomorphism $\overline{\varphi}:G\rightarrow SO(3)$. This finishes the proof.
	\end{proof}
	
	\begin{cor}\label{cor:existence}
		Let $\mathfrak{g}$ be a compact, semisimple Lie algebra with no simple factor isomorphic to $\mathfrak{so}(3)$. Then there exists an infinitesimally rigid Lie $\mathfrak{g}$-foliation $\mathcal{F}$ with dense leaves on some compact manifold $M$.
	\end{cor}
	
\subsection{An algorithm}
	We stress that Cor.~\ref{cor:existence} is not just an abstract existence result. Below, we give an explicit algorithm describing how the foliations in Cor.~\ref{cor:existence} can be constructed. 
	
	\vspace{0.2cm}
	
	\emph{Input:} A compact, connected, semisimple Lie group $G$ not admitting $SO(3)$ as a quotient.
	
	\emph{Output:} A finitely presented, dense subgroup $\Gamma\subset G$ that is a discrete Kazhdan group.
	
	\vspace{0.2cm}
	
	Instead of using Margulis' result \cite[Chapter~III, Prop.~7]{Margulis} as in the proof of Thm.~\ref{thm:no-so3}, our approach is based on the paper \cite{Cornulier}. This way, we can avoid the use of $p$-adic numbers and the theory of adèles. As a bonus, it is immediately clear that the obtained subgroup $\Gamma\subset G$ is not only finitely generated, but even finitely presented. So we can pick a compact, connected $4$-manifold $B$ with $\pi_1(B)=\Gamma$. The suspension foliation of the inclusion $\pi_1(B)\hookrightarrow G$ is an infinitesimally rigid Lie $\mathfrak{g}$-foliation with dense leaves whose existence is claimed in Cor.~\ref{cor:existence}.
	
	\vspace{0.2cm}
	\underline{\emph{Step 1:}} We may assume that $G$ is simple.
	
	Indeed, let $p:\widetilde{G}\rightarrow G$ be the simply connected covering of $G$. Then $\widetilde{G}=H_1\times\cdots\times H_k$ is a product of compact, connected, simple Lie groups $H_i$. If we have subgroups $\Gamma_i\subset H_i$ that are finitely presented, dense and have property (T) as discrete groups, then the same holds for the subgroup $p\left(\Gamma_1\times\cdots\times\Gamma_k\right)\subset G$. Let us check this in detail:
	\begin{itemize}
		\item Clearly, we have that $\Gamma_1\times\cdots\times\Gamma_k$ is dense in $H_1\times\cdots\times H_k$. Since $p:\widetilde{G}\rightarrow G$ is continuous and surjective, it follows that $p\left(\Gamma_1\times\cdots\times\Gamma_k\right)$ is dense in $G$.
		\item The product $\Gamma_1\times\cdots\times\Gamma_k$ has property (T) by \cite[Prop.~1.7.8]{Kazhdan}. It follows from \cite[Thm.~1.3.4]{Kazhdan} that also $p\left(\Gamma_1\times\cdots\times\Gamma_k\right)$ has property (T).
		\item By \cite[Chapter~4, Prop.~4]{groups}, the product $\Gamma_1\times\cdots\times\Gamma_k$ has a finite presentation $\langle S | R\rangle$. Because $\ker(p)\subset\widetilde{G}$ is closed and discrete, it must be finite since $\widetilde{G}$ is compact. Hence also the set $\ker\left(p|_{\Gamma_1\times\cdots\times\Gamma_k}\right)$ is finite. For every $g\in\ker\left(p|_{\Gamma_1\times\cdots\times\Gamma_k}\right)$, we pick a word $W_g$ in the free group $F(S)$ that represents $g$. It follows that
		\[
		\big\langle S| R\cup\left\{W_g:g\in\ker\left(p|_{\Gamma_1\times\cdots\times\Gamma_k}\right)\right\}\big\rangle
		\]
		is a finite presentation for $p\left(\Gamma_1\times\cdots\times\Gamma_k\right)\cong \Gamma_1\times\cdots\times\Gamma_k/\ker\left(p|_{\Gamma_1\times\cdots\times\Gamma_k}\right)$.
	\end{itemize}
	From now on, $G$ is a compact, connected, simple Lie group not locally isomorphic to $SO(3)$.
	
	\vspace{0.2cm}
	\underline{\emph{Step 2:}} We may assume that the center $Z(G)$ is trivial.
	
	The center $Z(G)$ is a finite normal subgroup of $G$ and the quotient $G/Z(G)$ has trivial center, see \cite[Ex.~7.11(b)]{Fulton}. Denote by $p:G\rightarrow G/Z(G)$ the quotient map. If $\Gamma\subset G/Z(G)$ is a finitely presented, dense subgroup with property (T) as a discrete group, then the same holds for the subgroup $p^{-1}(\Gamma)\subset G$. Let us check this in detail.
	\begin{itemize}
		\item We have that $p^{-1}(\Gamma)\subset G$ is dense because the projection $p:G\rightarrow G/Z(G)$ is an open map. If $U\subset G$ is open, then $p(U)\subset G/Z(G)$ is also open hence $p(U)\cap\Gamma$ is non-empty. So there exists $g\in U$ such that $p(g)\in\Gamma$, hence $g\in U\cap p^{-1}(\Gamma)$.
		\item Now consider the short exact sequence
		\[
		1\rightarrow p^{-1}(\Gamma)\cap Z(G)\rightarrow p^{-1}(\Gamma)\rightarrow \Gamma\rightarrow 1.
		\]
		First, the discrete group $\Gamma$ has property (T) by assumption and the discrete group $p^{-1}(\Gamma)\cap Z(G)$ has property (T) because it is finite \cite[Prop.~1.1.5]{Kazhdan}. Since property (T) is preserved by group extensions \cite[Prop.~1.7.6]{Kazhdan}, also $p^{-1}(\Gamma)$ has property (T). Similarly, the group $\Gamma$ is finitely presented by assumption. The group $p^{-1}(\Gamma)\cap Z(G)$ is finite, hence finitely presented \cite[Chapter~4, Prop.~1]{groups}. Since the property of being finitely presented is closed under group extensions \cite[Chapter~10, Cor.~2]{groups}, it follows that also $p^{-1}(\Gamma)$ is finitely presented.
	\end{itemize}
	From now on, $G$ is a compact, connected and simple Lie group with $Z(G)=1$, not locally isomorphic to $SO(3)$. We now proceed with the construction of $\Gamma\subset G$. Since we will use the book \cite{Witte} by Witte Morris, we first check that we are in the situation of \cite[Def.~5.1.2]{Witte}.
	
	\vspace{0.2cm}
	\underline{\emph{Step 3:}} $G$ is a closed subgroup of $SL(n,\mathbb{R})$ that is defined over $\mathbb{Q}$.
	
	To realize $G$ as a closed subgroup of $SL(n,\mathbb{R})$, we use the adjoint action $\text{Ad}:G\rightarrow\text{Aut}(\mathfrak{g})$. Its kernel $Z(G)$ is assumed to be trivial, hence $\text{Ad}$ embeds $G$ as a closed subgroup of $\text{Aut}(\mathfrak{g})$. Moreover, since $G$ is unimodular and connected, we know that $\det\text{Ad}(g)=1$ for all $g\in G$.  So we can view $G\cong\text{Ad}(G)$ a closed subgroup of $SL(n,\mathbb{R})$ after choosing any basis for $\mathfrak{g}$. 
	
In fact, we can describe $\text{Ad}(G)$ more explicitly. Indeed, we claim that $\text{Ad}(G)=\text{Aut}(\mathfrak{g})^{0}$ where $\text{Aut}(\mathfrak{g})^{0}\subset\text{Aut}(\mathfrak{g})$ denotes the identity component. This equality follows from well-known results in Lie theory that can be found in \cite[Chapter I, \S 14]{Knapp}. In more detail, on one hand we have that $\text{Ad}(G)\subset\text{Aut}(\mathfrak{g})$ is a connected Lie subgroup with Lie algebra $\text{ad}(\mathfrak{g})$. On the other hand, $\text{Aut}(\mathfrak{g})^{0}\subset\text{Aut}(\mathfrak{g})$ is a connected Lie subgroup with Lie algebra $\text{Der}(\mathfrak{g})$. Since $\mathfrak{g}$ is simple, we have $\text{Der}(\mathfrak{g})=\text{ad}(\mathfrak{g})$ and therefore we get $\text{Ad}(G)=\text{Aut}(\mathfrak{g})^{0}$.
	
	We now argue that $G$ is defined over $\mathbb{Q}$. That is, we need to show that there exists a subset $\mathcal{Q}$ of the polynomial ring $\mathbb{Q}[x_{1,1},\ldots,x_{n,n}]$ such that
	\[
	\text{Var}(\mathcal{Q}):=\{g\in SL(n,\mathbb{R}): Q(g)=0,\ \forall Q\in\mathcal{Q}\}
	\]
	is a subgroup of $SL(n,\mathbb{R})$ and $G=\text{Var}(\mathcal{Q})^{0}$.  
	 As explained above, we actually need to show that $\text{Aut}(\mathfrak{g})^{0}$ is defined over $\mathbb{Q}$. To this end, recall that the Lie algebra $\mathfrak{g}$ has a $\mathbb{Q}$-form \cite[Rem.~13]{Cornulier}, i.e. it has a basis whose associated structure constants lie in $\mathbb{Q}$. It follows that there exists a subset $\mathcal{Q}$ of $\mathbb{Q}[x_{1,1},\ldots,x_{n,n}]$ such that 
	\[
	\text{Aut}(\mathfrak{g})=\{g\in GL(n,\mathbb{R}): Q(g)=0,\ \forall Q\in\mathcal{Q}\}.
	\]
	In particular, $\text{Var}(\mathcal{Q})=SL(n,\mathbb{R})\cap \text{Aut}(\mathfrak{g})$ is a subgroup of $SL(n,\mathbb{R})$. Moreover, we have
	\[
	\text{Aut}(\mathfrak{g})^{0}=\big(SL(n,\mathbb{R})\cap \text{Aut}(\mathfrak{g})\big)^{0}=\text{Var}(\mathcal{Q})^{0}.
	\]
	This shows that $\text{Aut}(\mathfrak{g})^{0}$ is defined over $\mathbb{Q}$. 
	In conclusion, we now showed that one can view $G$ as a closed subgroup of  $SL(n,\mathbb{R})$ that is defined over $\mathbb{Q}$.
	
	\vspace{0.2cm}
	\underline{\emph{Step 4:}} Construction of the desired subgroup $\Gamma\subset G$.
	
	Take a number field $K\subset\mathbb{R}$ of degree $3$ that is not totally real, for instance $K:=\mathbb{Q}(\sqrt[3]{2})$. It has $3$ distinct field embeddings $K\hookrightarrow\mathbb{C}$ corresponding to the roots of $X^{3}-2=0$. They are given by $\text{Id},\sigma$ and its complex conjugate $\overline{\sigma}$, where $\sigma:K\hookrightarrow\mathbb{C}$ is given by 
	\[
	\begin{cases}
		\sigma(x)=x\ \text{for}\ x\in\mathbb{Q}\\
		\sigma(\sqrt[3]{2})=\sqrt[3]{2}e^{2\pi i/3}
	\end{cases}.
	\]
	The ring of integers of $K$ is $\mathbb{Z}[\sqrt[3]{2}]$. We now apply a well-known construction of lattices by Borel-Harish-Chandra \cite{harish}. We follow \cite[Prop.~5.5.8]{Witte}, which is a reformulation of the latter.
	
	In Step 3, we found a set of polynomials $\mathcal{Q}\subset\mathbb{Q}[x_{1,1},\ldots,x_{n,n}]$ such that $G=\text{Var}(\mathcal{Q})^{0}$. Consider the complexification
	\[
	\text{Var}_{\mathbb{C}}(\mathcal{Q}):=\{g\in SL(n,\mathbb{C}): Q(g)=0,\ \forall Q\in \mathcal{Q}\},
	\]
	and define the Galois conjugate $G^{\sigma}$ of $G$ to be the identity component $G^{\sigma}:=\text{Var}_{\mathbb{C}}(\mathcal{Q})^{0}$. This is a complex Lie group with Lie algebra $\mathfrak{g}\otimes\mathbb{C}$, see \cite[\S~18.1.3 and \S~18.1.8]{Witte}. We can now recall the result \cite[Prop.~5.5.8]{Witte}. It states that the natural map
	\[
	\Delta:K\rightarrow\mathbb{R}\times\mathbb{C}:\gamma\mapsto(\gamma,\sigma(\gamma))
	\]
	embeds a finite index subgroup $\dot{G}(\mathbb{Z}[\sqrt[3]{2}])\subset G(\mathbb{Z}[\sqrt[3]{2}])$ as an irreducible lattice in $G\times G^{\sigma}$. Here $G(\mathbb{Z}[\sqrt[3]{2}])$ consists of the elements of $G$ whose entries lie in the ring of integers $\mathbb{Z}[\sqrt[3]{2}]$. At last, we can define the desired subgroup $\Gamma\subset G$ by setting
	$$
	\Gamma:=\dot{G}(\mathbb{Z}[\sqrt[3]{2}]).
	$$
	Let us check that $\Gamma\subset G$ is indeed finitely presented, dense and a discrete Kazhdan group.
	\begin{itemize}
		\item Since $G$ is compact, the lattice $\Delta(\Gamma)\subset G\times G^{\sigma}$ is cocompact by \cite[Cor.~5.5.10]{Witte}. As cocompact lattices in connected Lie groups are finitely presented \cite[Thm.~6.15]{Rag}, it follows that $\Delta(\Gamma)$ is finitely presented. Hence the same holds for $\Gamma$.
		\item The fact that $\Gamma\subset G$ is dense follows from irreducibility of the lattice $\Delta(\Gamma)\subset G\times G^{\sigma}$. By definition \cite[Def.~4.3.1]{Witte}, this means that $\Delta(\Gamma)N$ is dense in $G\times G^{\sigma}$ for every non-compact, closed, normal subgroup $N\subset G\times G^{\sigma}$. In particular, we can take $N$ to be $\{1\}\times G^{\sigma}$. Indeed, $G^{\sigma}$ is not compact because it is of finite index in $\text{Var}_{\mathbb{C}}(\mathcal{Q})$. If $G^{\sigma}$ were compact then also $\text{Var}_{\mathbb{C}}(\mathcal{Q})$ would be compact, which is impossible since $\text{Var}_{\mathbb{C}}(\mathcal{Q})$ is a complex affine variety. So irreducibility of $\Delta(\Gamma)\subset G\times G^{\sigma}$ implies that 
		\[
		\Delta(\Gamma)(\{1\}\times G^{\sigma})\subset G\times G^{\sigma}
		\]
		is dense. Since the first projection $\text{pr}_1:G\times G^{\sigma}\rightarrow G$ is continuous and surjective, we get that
		$
		\text{pr}_1\big(\Delta(\Gamma)(\{1\}\times G^{\sigma})\big)=\Gamma
		$
		is a dense subgroup of $G$.
		\item At last, we check that $\Gamma$ has property (T) as a discrete group. Equivalently, we show that $\Delta(\Gamma)$ has property (T) as a discrete group. Since property (T) is inherited by lattices \cite[Thm.~1.7.1]{Kazhdan}, it suffices to argue that $G\times G^{\sigma}$ has property (T). Recall that compact groups have property (T) by \cite[Prop.~1.1.5]{Kazhdan} and that the product of property (T) groups has property (T) by \cite[Prop.~1.7.8]{Kazhdan}. Hence, it remains to show that $G^{\sigma}$ has property (T). First note that $G^{\sigma}$ is an almost $\mathbb{C}$-simple algebraic group, meaning that the only proper
		algebraic normal subgroups $N$ of $G^{\sigma}$ defined over $\mathbb{C}$ are finite. Indeed, the Lie algebra $\mathfrak{n}$ of such a subgroup $N$ is an ideal in the complex Lie algebra $\mathfrak{g}\otimes\mathbb{C}$ of $G^{\sigma}$, see \cite[\S~10.4, Cor.~A]{Humphreys}. However, the complexification $\mathfrak{g}\otimes\mathbb{C}$ of the compact, simple Lie algebra $\mathfrak{g}$ is again simple \cite{Kirillov}. So either $\mathfrak{n}=\mathfrak{g}\otimes\mathbb{C}$ and therefore $N=G^{\sigma}$, or $\mathfrak{n}=0$ and therefore $N$ is finite since $N$ is algebraic. This shows that $G^{\sigma}$ is almost $\mathbb{C}$-simple. It follows that, if $G^{\sigma}$ did not have property (T), then necessarily it would be locally isomorphic to $SL(2,\mathbb{C})$ by \cite[Thm.~1.6.1]{Kazhdan} and \cite[Rem.~1.6.3]{Kazhdan}. At the Lie algebra level, we get that $\mathfrak{sl}(2,\mathbb{C})\cong\mathfrak{g}\otimes\mathbb{C}$, i.e. $\mathfrak{g}$ is a real form for the complex Lie algebra $\mathfrak{sl}(2,\mathbb{C})$. So either $\mathfrak{g}\cong\mathfrak{sl}(2,\mathbb{R})$ or $\mathfrak{g}\cong\mathfrak{su}(2)$. The first case is impossible since $\mathfrak{g}$ is compact and $\mathfrak{sl}(2,\mathbb{R})$ is not. If $\mathfrak{g}\cong\mathfrak{su}(2)\cong\mathfrak{so}(3)$ then $G$ is locally isomorphic with $SO(3)$, which we excluded. So we reach a contradiction, which shows that $G^{\sigma}$ has property (T). Hence $\Gamma$ has property (T) as well.

	\end{itemize}
	

	\color{black}
	
	\subsection{A concrete example}\label{sec:example}
	
	The aim of this subsection is to work out a concrete example of Thm.~\ref{thm:main-result}. It deviates slightly from the algorithm described above but uses similar ideas. The main task is to construct a discrete Kazhdan group $\Gamma$ which can serve as the fundamental group $\pi_1(B)$ in Thm.~\ref{thm:main-result}. We will obtain $\Gamma$ as a suitable arithmetic lattice in $SO(3,2)$ which can be realized as a dense subgroup of $SO(5)$.  Our interest in this example stems from the fact that the lattice in question was used to solve the Ruziewicz problem for $n\geq 4$, showing that the Lebesgue measure is the only finitely additive measure defined on all Lebesgue measurable subsets of the sphere $S^{n}$ that is invariant under the action of $SO(n+1)$ and of total measure $1$, see \cite[Thm.~3.4.2]{Lubotzky}.

	\subsubsection{Arithmetic lattices}
	We start by recalling an arithmetic construction of lattices, following \cite[Chapter IX, \S 1.7]{Margulis}.
	Fix an integer $n\geq 3$ and a subfield $K\subset\mathbb{C}$ which is a finite extension of $\mathbb{Q}$. Denote by $d=[K:\mathbb{Q}]$ the degree of the extension. Assume that 
	\[
	\Phi=\sum_{1\leq i,j\leq n}a_{ij}x_ix_j
	\]
	is a non-degenerate quadratic form in $n$ variables with coefficients in $K$. As usual, $a_{ij}=a_{ji}$. Recall that there are $d$ distinct field embeddings $\sigma_1,\ldots,\sigma_d:K\rightarrow\mathbb{C}$ with $\sigma_1=\text{Id}$. We call such an embedding $\sigma$ real if $\sigma(K)\subset\mathbb{R}$, otherwise it is called imaginary. Two embeddings of $K$ into $\mathbb{C}$ are called equivalent if one is the complex conjugate of the other. Choosing a representative in each equivalence class of embeddings, we obtain the set
	\[
	\mathcal{R}:=\{\sigma_1,\ldots,\sigma_l\}.
	\]
	For every $\sigma\in\mathcal{R}$, we define $k_{\sigma}:=\mathbb{R}$ if $\sigma$ is real and $k_{\sigma}:=\mathbb{C}$ if $\sigma$ is imaginary. Each field embedding $\sigma\in\mathcal{R}$ gives rise to a new quadratic form
	\begin{equation}\label{eq:sigma-phi}
		\Phi_{\sigma}=\sum_{1\leq i,j\leq n}\sigma(a_{ij})x_ix_j
	\end{equation}
	and we set
	\[
	\mathcal{T}=\{\sigma\in\mathcal{R}:\ k_{\sigma}=\mathbb{R}\ \text{and}\ \Phi_{\sigma}\ \text{is either positive definite or negative definite}\}.
	\]
	Choose a subset $\mathcal{S}\subset\mathcal{R}$ such that $\mathcal{R}\setminus\mathcal{T}\subset\mathcal{S}$. Denote by $SO_{\Phi}$ the special orthogonal group of the quadratic form $\Phi$, consisting of complex matrices with determinant $1$ preserving $\Phi$. It is an algebraic group defined over $K$. We let $\mathcal{O}$ be the ring of integers in $K$ and denote by $SO_{\Phi}(\mathcal{O})$ the group of $\mathcal{O}$-valued matrices in $SO_{\Phi}$. Similarly, we define for each $\sigma\in \mathcal{S}$ the group $SO_{\Phi_{\sigma}}(k_{\sigma})$ of $k_{\sigma}$-valued matrices with determinant $1$ preserving the quadratic form $\Phi_{\sigma}$ defined in \eqref{eq:sigma-phi}.
	In the following, we identify $SO_{\Phi}(\mathcal{O})$ with its image under the embedding
	\[
	\prod_{\sigma\in \mathcal{S}}\sigma:SO_{\Phi}(\mathcal{O})\rightarrow\prod_{\sigma\in \mathcal{S}}SO_{\Phi_{\sigma}}(k_{\sigma}).
	\]
	\begin{prop}\cite[Chapter IX, \S 1.7]{Margulis}\label{prop:lattice}
		Assume that $\mathcal{T}\neq\mathcal{R}$ and that the group $SO_{\Phi}$ is almost $K$-simple. Then $SO_{\Phi}(\mathcal{O})$ is an irreducible arithmetic lattice in $\prod_{\sigma\in \mathcal{S}}SO_{\Phi_{\sigma}}(k_{\sigma})$.
	\end{prop}
	
	Almost $K$-simplicity means that proper algebraic $K$-closed normal subgroups of $SO_{\Phi}$ are finite. This condition is equivalent to the following: either $n\neq 4$ or $n=4$ and the discriminant of the quadratic form $\Phi$ is not a square in $K$. The lattice $SO_{\Phi}(\mathcal{O})$ is non-cocompact if the quadratic form $\Phi$ represents $0$ over $K$ non-trivially, and cocompact otherwise.

	\subsubsection{An example}
	
	We let $n=5$ and consider the field extension $\mathbb{Q}\subset\mathbb{Q}(\sqrt{2})$ of degree $2$. Consider the non-degenerate quadratic form
	\[
	\Phi=x_1^{2}+x_2^{2}+x_3^{2}-\sqrt{2}x_4^{2}-\sqrt{2}x_5^{2}
	\]
	in $5$ variables with coefficients in $\mathbb{Q}(\sqrt{2})$. There are $2$ field embeddings $\sigma_1,\sigma_2:\mathbb{Q}(\sqrt{2})\rightarrow\mathbb{C}$, namely $\sigma_1=\text{Id}$ and $\sigma_2$ given by
	\begin{equation}\label{eq:galois}
		\sigma_2(a+b\sqrt{2})=a-b\sqrt{2},\hspace{0.5cm}a,b\in\mathbb{Q}.
	\end{equation}
	Both are real, and we have $\mathcal{R}=\{\sigma_1,\sigma_2\}$. The embedding $\sigma_2$ defines a new quadratic form
	\[
	\Phi_{\sigma_2}=x_1^{2}+x_2^{2}+x_3^{2}+\sqrt{2}x_4^{2}+\sqrt{2}x_5^{2},
	\]
	and in this case we have $\mathcal{T}=\{\sigma_2\}$. Let us choose $\mathcal{S}=\{\sigma_1,\sigma_2\}$. Applying Prop.~\ref{prop:lattice} gives an arithmetic lattice 
	\[
	\Gamma:=SO_{\Phi}(\mathbb{Z}[\sqrt{2}])\subset SO(3,2)\times SO(5).
	\]
	The first projection in $SO(3,2)\times SO(5)$ realizes $\Gamma$ as a lattice in $SO(3,2)$, while the second projection in $SO(3,2)\times SO(5)$ realizes $\Gamma$ as a subgroup of the compact Lie group $SO(5)$.
	We now check that this lattice can be used to construct an example of Thm.~\ref{thm:main-result}.
	
	\begin{prop}
		Let $\Gamma$ be the discrete group defined above. It satisfies the following:
		\begin{enumerate}
			\item $\Gamma$ is finitely presented.
			\item The second projection $\text{pr}_{2}:\Gamma\rightarrow SO(5)$ has dense image.
			\item $\Gamma$ has property (T) as a discrete group.
		\end{enumerate}
	\end{prop}
	\begin{proof}
		$(1)$ We view $\Gamma$ as a subgroup of $SO(3,2)$. Let $SO(3,2)^{+}$ be the identity component of $SO(3,2)$, which is a normal subgroup of index $2$, see \cite[Prop.~1.145]{Knapp}. The property of being finitely presented is closed under group extensions \cite[Chapter~10, Cor.~2]{groups}, hence it follows that $\Gamma$ is finitely presented when $\Gamma\cap SO(3,2)^{+}$ and $\Gamma/\Gamma\cap SO(3,2)^{+}$ are. For the latter, note that
		\[
		\frac{\Gamma}{\Gamma\cap SO(3,2)^{+}}\cong\frac{\Gamma SO(3,2)^{+}}{SO(3,2)^{+}}=\frac{SO(3,2)}{SO(3,2)^{+}}
		\]
		has order $2$, hence it has a finite presentation $\langle a | a^{2}\rangle$. To see that $\Gamma\cap SO(3,2)^{+}$ is finitely presented, we argue as follows. It is known that cocompact lattices in connected Lie groups are finitely presented, see \cite[Thm.~6.15]{Rag}. So we only have to argue that $\Gamma\cap SO(3,2)^{+}$ is a cocompact lattice in $SO(3,2)^{+}$. Because $SO(3,2)^{+}\subset SO(3,2)$ is an open subgroup, it suffices to show that $\Gamma\subset SO(3,2)$ is a cocompact lattice, see \cite[\S~2.C.]{caprace}. To check that the lattice $\Gamma\subset SO(3,2)$ is cocompact, we only need to argue that the quadratic form $\Phi$ cannot represent $0$ over $\mathbb{Q}(\sqrt{2})$ non-trivially. To do so, assume that $x_1,\ldots,x_5\in\mathbb{Q}(\sqrt{2})$ satisfy
		\[
		x_1^{2}+x_2^{2}+x_3^{2}-\sqrt{2}x_4^{2}-\sqrt{2}x_5^{2}=0.
		\]
		Applying the Galois automorphism $\sigma_2$ of $\mathbb{Q}(\sqrt{2})$ defined in \eqref{eq:galois}, we get
		\[
		\big(\sigma_2(x_1)\big)^{2}+\big(\sigma_2(x_2)\big)^{2}+\big(\sigma_2(x_3)\big)^{2}+\sqrt{2}\big(\sigma_2(x_4)\big)^{2}+\sqrt{2}\big(\sigma_2(x_5)\big)^{2}=0.
		\]
		It follows that $\sigma_2(x_1)=\cdots=\sigma_2(x_5)=0$ and therefore $x_1=\cdots=x_5=0$. Hence the quadratic form $\Phi$ does not represent $0$ over $\mathbb{Q}(\sqrt{2})$ non-trivially, and therefore $\Gamma\subset SO(3,2)$ is cocompact. Altogether, this proves that the group $\Gamma$ is finitely presented.
		
		$(2)$ This was proved in \cite[Prop.~3.4.3]{Lubotzky}.

		$(3)$ We view $\Gamma$ as a lattice inside $SO(3,2)$. Because property (T) is inherited by lattices \cite[Thm.~1.7.1]{Kazhdan}, it suffices to show that $SO(3,2)$ has property (T). As before, we denote by $SO(3,2)^{+}\subset SO(3,2)$ the identity component, which is a closed normal subgroup. It follows that $SO(3,2)$ has property (T) as soon as we show that $SO(3,2)^{+}$ and $SO(3,2)/SO(3,2)^{+}$ have property (T), see \cite[Prop.~1.7.6]{Kazhdan}. First, the identity component $SO(3,2)^{+}$ has property (T) since it is a connected simple Lie group with Lie algebra not of the form $\mathfrak{so}(n,1)$ or $\mathfrak{su}(n,1)$, see \cite[Thm.~3.5.4]{Kazhdan}. Next, the quotient $SO(3,2)/SO(3,2)^{+}$ is discrete because $SO(3,2)^{+}\subset SO(3,2)$ is open. As finite discrete groups are compact, $SO(3,2)/SO(3,2)^{+}$ has property (T) by \cite[Prop.~1.1.5]{Kazhdan}. We conclude that $\Gamma$ has property (T).
	\end{proof}
	
	\color{black}

	Fix a compact, connected manifold $4$-dimensional manifold $B$ with $\pi_1(B)=\Gamma$. Applying Thm.~\ref{thm:main-result}, the suspension of $\text{pr}_{2}:\pi_1(B)\rightarrow SO(5)$ gives an infinitesimally rigid Lie foliation $\mathcal{F}$ with dense leaves on the compact manifold $B_{\text{pr}_{2}}$.

	\begin{remark}
		The above argument works for all $n\geq 5$, considering the quadratic form
		\[
		\Phi=x_1^{2}+\cdots+x_{n-2}^{2}-\sqrt{2}x_{n-1}^{2}-\sqrt{2}x_{n}^{2}.
		\]
		We get an arithmetic lattice $\Gamma\subset SO(n-2,2)\times SO(n)$ with property (T) which projects to a dense subgroup in $SO(n)$, see \cite[Prop.~3.4.3]{Lubotzky}.
		The method fails for $n<5$. Actually, Thm.~\ref{thm:main-result} admits no examples in which $G=SO(n)$ with $n<5$. Clearly, $SO(2)$ is excluded by Prop.~\ref{prop:obstruction} because it is abelian hence its Lie algebra is not perfect. The fact that $SO(3)$ and $SO(4)$ are impossible follows from Thm.~\ref{thm:no-so3} $(2)$. To see that $SO(3)$ is a quotient of $SO(4)$, recall the isomorphism of Lie algebras $\mathfrak{so}(4)\cong\mathfrak{so}(3)\oplus\mathfrak{so}(3)$ and use Lemma~\ref{lem:rephrase}.
		
	\end{remark}

	\end{document}